\documentclass[12pt]{article}

\usepackage[dvips,colorlinks,bookmarksopen,bookmarksnumbered,citecolor=red,urlcolor=red]{hyperref}

\usepackage{amssymb}
\usepackage{amsthm}
\usepackage{amsmath}
\usepackage{enumerate}
\usepackage{algorithm}
\usepackage{algorithmic}
\usepackage{epsfig,epic,subfigure}
\usepackage{graphicx}
\usepackage{graphics}
\usepackage{geometry}

\newcommand\BibTeX{{\rmfamily B\kern-.05em \textsc{i\kern-.025em b}\kern-.08em
T\kern-.1667em\lower.7ex\hbox{E}\kern-.125emX}}

\newenvironment{psmallmatrix}{\left[\begin{smallmatrix}}{\end{smallmatrix}\right]}
\renewenvironment{pmatrix}{\begin{psmallmatrix}}{\end{psmallmatrix}}

\newtheorem{theorem}{Theorem}[section]
\newtheorem{lemma}[theorem]{Lemma}
\theoremstyle{definition}

\newtheorem*{assumption}{Assumption}
\theoremstyle{remark}
\newtheorem{remark}[theorem]{Remark}

\newcommand{\diag}{{\rm diag}}

\def \uu {{\mathbf u}}
\def \vv {{\mathbf v}}

\begin{document}


\title{Semi-Decentralized Approximation of Optimal Control of Distributed Systems Based on a Functional Calculus}

\author{Y.~Yakoubi and M. ~Lenczner \\ Department Time-Frequency, FEMTO-ST Institute, \\ 26, Rue de
l'Epitaphe, 25030 Besan\c{c}on, FRANCE.}
%
%





\begin{abstract}
This paper discusses a new approximation method for operators which
are solution to an operational Riccati equation (\textbf{ORE}). The
latter is derived from the theory of optimal control of linear
problems posed in Hilbert spaces. The approximation is based on the
functional calculus of self-adjoint operators and the Cauchy
formula. Under a number of assumptions the approximation is suitable
for implementation on a semi-decentralized computing architecture in
view of real-time control. Our method is particularly applicable to
problems in optimal control of systems governed by partial
differential equations with distributed observation and control.
Some relatively academic applications are presented for
illustration. More realistic examples relating to microsystem arrays
have already been published.
\end{abstract}


\maketitle


\section{Introduction}
This work is a contribution to the area of semi-decentralized
optimal control of large linear distributed systems for real-time
applications. It applies to systems modeled by linear partial
differential equations with observation and control distributed over
the whole domain. This is a strong assumption, but it does not mean
that actuators and sensors are actually continuously distributed.
The models satisfying such assumption may be derived by
homogenization of systems with periodic distribution of actuators
and sensors.

\medskip

In this paper we consider two classes of systems, those with bounded
control
and bounded observation operators as in R. Curtain and H. Zwart \cite{CurZwa}%
, and those with unbounded control but bounded observation operators
as in H.T. Banks and K. Ito \cite{BanIto}. In an example, we show
how the method may also be applied to a particular boundary control
problem. We view possible applications in the field of systems
including a network of actuators and sensors, see for instance
\cite{HuiYak} dedicated to arrays of Atomic Force Microscopes.

\medskip

We consider four linear operators $A,$ $B,$ $C$, $S$, and the Linear
Quadratic Regulator (LQR) problem stated classically as a
minimization problem,
\begin{eqnarray}
\mathcal{J}\left( z_{0},\uu\right) &=& \min\limits_{u\in
U}\mathcal{J}\left(
z_{0},u\right) ,  \label{minimization} \\
\text{with} ~ \mathcal{J}\left( z_{0},u\right) &=& \int_{0}^{+\infty
}\left\| Cz\right\| _{Y}^{2}+\left( Su,u\right) _{U} ~ dt,
\label{functional}
\end{eqnarray}%
constrained by a state equation,
\begin{equation}
\frac{dz}{dt}\left( t\right) =Az\left( t\right) +Bu\left( t\right)
\quad \text{for }t>0\quad \text{and }z\left( 0\right) =z_{0}\text{.}
\label{state equation}
\end{equation}%
Under usual assumptions there exists a unique solution
$\uu=-S^{-1}B^{\ast
}Pz$, where $P$ is a solution of the\ operational Riccati equation (\textbf{%
ORE}),
\begin{equation}
A^{\ast }P+PA-PBS^{-1}B^{\ast }P+C^{\ast }C=0\text{.} \label{Riccati
eq}
\end{equation}%
In the framework of \cite{CurZwa}, $A:D(A)\subset Z\mapsto Z,$
$B:U\mapsto Z, $ $C:Z\mapsto Y$, $S:U\mapsto U$ and consequently
$P:Z\mapsto Z$ for some linear spaces $Z,$ $U$ and $Y$. To derive
our semi-decentralized realization
of $Pz$, we further assume that there exists a linear self-adjoint operator $%
\Lambda :X\mapsto X$, three one-to-one mappings
\begin{equation}
\Phi _{Z}:X^{n_{Z}}\mapsto Z,\quad \Phi _{U}:X^{n_{U}}\mapsto U\quad \text{%
and }\Phi _{Y}:X^{n_{Y}}\mapsto Y,  \label{isomorphisms}
\end{equation}%
with appropriate integers $n_{Z},$ $n_{U}$ and $n_{Y},$ and four
continuous matrix-valued functions $\lambda \mapsto a(\lambda ),$
$\lambda \mapsto b(\lambda )$, $\lambda \mapsto c(\lambda )$ and
$\lambda \mapsto s(\lambda )$ such that
\begin{equation}
A=\Phi _{Z}a(\Lambda )\Phi _{Z}^{-1},\quad B=\Phi _{Z}b(\Lambda
)\Phi _{U}^{-1},\quad C=\Phi _{Y}c(\Lambda )\Phi _{Z}^{-1}\quad
\text{and }S=\Phi _{U}s(\Lambda )\Phi _{U}^{-1}.  \label{operator
factorizations}
\end{equation}
We notice that the functions of the self-adjoint operator $\Lambda $
used in the above formulae are defined using spectral theory of
self-adjoint operators (having a real spectrum) with compact or not
compact resolvent so that to encompass bounded and unbounded
domains. From (\ref{operator factorizations}), it follows that the
Riccati operator $P$ is factorized as
\begin{equation}
P=\Phi _{Z}p(\Lambda )\Phi _{Z}^{-1},  \label{decomposition of P}
\end{equation}%
where $\lambda \mapsto p(\lambda )$ is a continuous function,
solution of the algebraic Riccati equation (\textbf{ARE})
\begin{equation}
a^{T}\left( \lambda \right) p+pa\left( \lambda \right) -pb\left(
\lambda \right) s^{-1}\left( \lambda \right) b^{T}\left( \lambda
\right) p+c^{T}\left( \lambda \right) c\left( \lambda \right)
=0\text{.} \label{algebraic Ricc equation}
\end{equation}%
Our goal is reached once separate efficient semi-decentralized
approximations of $\Phi _{Z}$, $p(\Lambda )$ and $\Phi _{Z}^{-1}$
are provided for the realization of $P$ through (\ref{decomposition
of P}). This is generally not an issue for $\Phi _{Z}$ and for $\Phi
_{Z}^{-1},$ then the point is the semi-decentralized approximation
of $p(\Lambda )$. It might be
build by a polynomial approximation,%
\begin{equation}
p_{N}(\Lambda )=\sum_{k=0}^{N}d_{k}\Lambda ^{k}, \label{polynomial
approximation}
\end{equation}%
or a rational approximation,
\begin{equation}
\displaystyle p_{N}(\Lambda
)=\frac{\sum\limits_{k=0}^{N^{N}}d_{k}\Lambda
^{k}}{\sum\limits_{k^{\prime }=0}^{N^{D}}d_{k^{\prime }}^{\prime
}\Lambda ^{k^{\prime }}}\text{.}  \label{rational approximation}
\end{equation}%
Then, for practical implementations, the operator $\Lambda $ could
be replaced by a discretizations $\Lambda _{h},$ with parameter $h.$
We
emphasize that the formulae (\ref{polynomial approximation}) or (\ref%
{rational approximation}) yield large approximation errors, with respect to $%
h,$ due to the high powers of $\Lambda _{h}$. To overcome this
defect, we use an approximation based on the Cauchy integral which
requires to know the poles of $p$. In practice, we first approximate
the function $\lambda \mapsto p(\lambda )$ by a polynomial
approximation or a rational approximation $p_{N}(\lambda )$ with
degrees $N$ or $(N^{N},N^{D})$ sufficiently high to insure a very
small error. When $p_{N}$ is known its poles also, so we can state
the Cauchy formula for $p_{N}(\Lambda )$. This yields to introduce
the equations of the complex function $v=v_{1}+iv_{2}$
for each input $z\in Z,$%
\begin{equation}
(\xi -\Lambda )v=-i\xi ^{\prime }p_{N}\left( \xi \right) z,
\label{syst of DS}
\end{equation}%
where $\xi :(0,2\pi )\rightarrow \mathbb{C}$ is the contour of the
Cauchy formula. Denoting by $v^{\ell }$ the solution corresponding
to a quadrature point $\xi _{\ell }$ of the contour and $\omega
_{\ell }$ some quadrature weights, the final approximation of
$p(\Lambda )z$ is
\begin{equation}
p_{N,M}(\Lambda )z=\frac{1}{2\pi }\sum_{\ell =1}^{M}\omega _{\ell
}v_{1}^{\ell }.  \label{DS approximation}
\end{equation}%
Remark that the number $M$ of quadrature points is the only
important parameter governing the approximation error. For real-time
computation, the expression of $p_{N}$ is pre-computed, so the
approximation cost is also governed by $M$ only. With this method,
we do not observe a lack of precision when $\Lambda $ is replaced by
its discretizations $\Lambda _{h}$ and $M$ is large. In the sequel,
we show that the same derivation can be
done directly for $Qz=-S^{-1}B^{\ast }Pz$ provided that the isomorphisms $%
\Phi _{Z}$ and $\Phi _{U}$ are also some functions of $\Lambda .$

\medskip

This approach based on functional calculus is relatively simple, but
in each case it requires to determine the isomorphisms
(\ref{isomorphisms}). The theory has already been applied in
\cite{Yak} to a LQG control problem with a bounded operator $B$ that
is not a function of $\Lambda $. It has been shown how the control
approximation can be implemented through a distributed electronic
circuit. In \cite{LenYak1} and \cite{HuiYak} it has also been
applied to a one-dimensional array of cantilevers with regularly
spaced
actuators and sensors for which the operator $C$ is not a function of $%
\Lambda $. The underlying model was derived with a multiscale
method, an implementation of the semi-decentralized control was
provided in the form of a periodic network of resistors, and the
numerical validations of the complete strategy was carried out. In
the present paper, we illustrate the theory with four simpler
examples ranging from a simple heat equation with internal bounded
control and observation operators, a heat equation with an unbounded
control operator, a vibrating Euler-Bernoulli beam, and a heat
equation with a boundary controls.

\medskip

We notice that our method together improves and generalizes a
previous paper \cite{KadLenMr}. It was related to a specific
application, namely vibration control problem for a plate with a
periodic distribution of piezoelectric actuators and sensors. There,
the general isomorphisms (\ref{isomorphisms}) and the general
factorization (\ref{operator factorizations}) were not
introduced, and $p(\Lambda )$ was approximated by a polynomial as in (\ref%
{polynomial approximation}) which were severely limiting the
accuracy of the approximation. In both papers, the control method is
a LQR, but the theory is applicable to Riccati equations that may
arise in a number of other control problem, for instance for $H_{2}$
or $H_{\infty }$ dynamic compensators. Other extensions are also
possible, for instance, we may want to deal with functions of a non
self-adjoint operator $\Lambda $. In such a
case, another functional calculus, like these in \cite{MarSan} or in \cite%
{Haa}, could be used instead of the spectral theory. Other
frameworks for control problems of infinite dimensional systems
could also be used, for instance this of \cite{Lasiecka} for optimal
control with unbounded observations and unbounded controls.

\medskip

Other techniques have already been established, see \cite{BamPag}, \cite%
{PagBam}, \cite{Jov}, \cite{DanDul}, \cite{LanDAn} and the
references therein. But they are mostly focused on the infinite
length systems, see \cite{BamPag}, \cite{PagBam}, \cite{Jov} and
\cite{LanDAn} for systems governed by partial differential
equations, and \cite{DanDul} for discrete systems. Finally, in
\cite{LenMonYak} we developed another theoretical framework based on
the \textit{diffusive realization }applicable to a broad range of
linear operators on bounded or unbounded domains. In principle this
approach allows to cover general distributed control problems with
internal or boundary control. However, in this first paper in the
subject, only one-dimensional domains and linear operational
equations (e.g. Lyapunov equations) are covered.

\medskip

The paper is organized as follows. Notations and basic definitions
are recalled in Section \ref{Notations}. In Section \ref{Bounded
Control operators} the abstract approximation method is stated in
the framework of bounded control and observation operators. The
framework of unbounded control operators is treated in Section
\ref{Unbounded control operators}. Some extensions are outlined in
Section \ref{Extensions}. Most proofs are concentrated in Section
\ref{Proofs}. The illustrative examples are detailed
in Section \ref{Application} and finally the paper is concluded by Section %
\ref{conclusion}.

\section{Preliminary Results and Notations\label{Notations}}

The norm and the inner product of an Hilbert space $E$ are denoted by $%
||.||_{E}$ and $(.,.)_{E}.$ For a second Hilbert space $F,$ $\mathcal{L}%
(E,F) $ denotes the space of continuous linear operators from $E$ to
$F.$ In
addition, $\mathcal{L}(E,E)$ is denoted by $\mathcal{L}(E).$ One says that $%
\Phi \in \mathcal{L}(E,F)$ is an isomorphism from $E$ to $F$ if
$\Phi $ is one-to-one and if its inverse is continuous.

\medskip

Since the approximation method of $P$ is based on the concept of
matrices of functions of a self-adjoint operator, this section is
devoted to their definition. Let $\Lambda $ be a self-adjoint
operator on a separable Hilbert space $X$ with domain $D(\Lambda )$,
we denote by $\sigma (\Lambda )$ its spectrum and by $I_{\sigma
}=(\sigma _{\min },\sigma _{\max })\subset \mathbb{R}$ an open
interval that includes $\sigma (\Lambda )$. We recall that if
$\Lambda $ is compact then $\sigma (\Lambda )$ is bounded and is
only constituted of eigenvalues $\lambda _{k}.$ They are the
solutions to the eigenvalue problem $\Lambda \phi _{k}=\lambda
_{k}\phi _{k}$ where $\phi _{k}$ is an eigenvector associated to
$\lambda _{k}$ chosen normed in $X$, i.e. such that $||\phi
_{k}||_{X}=1$. For a given real valued function $f$, continuous on
$I_{\sigma }$, $f(\Lambda )$ is the linear self-adjoint operator on
$X$ defined by
\begin{equation*}
f(\Lambda )z=\sum_{k=1}^{\infty }f(\lambda _{k})z_{k}\phi _{k}\quad \text{%
where }z_{k}=(z,\phi _{k})_{X},
\end{equation*}%
with domain $ D(f(\Lambda ))=\{z\in X~|~\sum\limits_{k=1}^{\infty
}\left\vert f(\lambda _{k})z_{k}\right\vert ^{2}~<\infty \}.$
Then, if $f$ is a $ n_{1}\times n_{2}$ matrix of real valued functions $f_{ij},$ continuous on $%
I_{\sigma }$, $f(\Lambda )$ is a matrix of linear operators
$f_{ij}(\Lambda ) $ with domain
\begin{equation*}
D(f(\Lambda ))=\{z\in X^{n_{2}}~|~\sum\limits_{k=1}^{\infty
}\sum\limits_{j=1}^{n_{2}}|f_{ij}(\lambda
_{k})(z_{j})_{k}|^{2}~<\infty \quad \forall i=1 \ldots n_{1}\}.
\end{equation*}

\medskip

In the general case, where $\Lambda $ is not compact and where $f$
is still a continuous function, the self-adjoint operator $f(\Lambda
)$ is defined on $X$ by the Stieltjes integral
\begin{equation*}
f(\Lambda )=\int_{-\infty }^{+\infty }\lambda ~ d E_{\lambda },
\end{equation*}
and its domain is $D(f(\Lambda )) = \{z\in X ~ |~ \int_{-\infty
}^{+\infty
}|f(\lambda )|^{2}\text{ }d||E_{\lambda }z||_{X}^{2} ~ <\infty \} $ where $%
E_{\lambda }$ is the spectral family associated to $\Lambda$, see \cite%
{DauLio}. When $f$ is a matrix, $f(\Lambda )$ is a matrix of linear
operators with entries defined by the above formula and with domain
$$ D(f(\Lambda ))=\{z\in X^{n_{2}} ~ | ~ \int_{-\infty }^{+\infty
}\sum\limits_{j=1}^{n_{2}}|f_{ij}(\lambda )|^{2} ~ d||E_{\lambda
}z_{j}||_{X}^{2} ~ <\infty \quad \forall i=1 \ldots n_{1}\}.$$


\section{Bounded Control Operators\label{Bounded Control operators}}

In this section, we state the approximation result in the framework
of bounded input operators. We follow the mathematical setting
\cite{CurZwa} of the LQR problem (\ref{minimization}-\ref{state
equation}). So, $A$ is the infinitesimal generator of a continuous
semigroup on a separable Hilbert
space $Z$ with dense domain $D(A)$, $B\in \mathcal{L}(U,Z)$, $C\in \mathcal{L%
}(Z,Y)$ and $S\in \mathcal{L}(U,U)$ where $U$ and $Y$ are two
Hilbert spaces. We assume that $(A,B)$ is stabilizable and that
$(A,C)$ is detectable, in the sense that there exist $Q\in
\mathcal{L}(Z,U)$ and $F\in \mathcal{L}(Y,Z)$ such that $A-BQ$ and
$A-FC$ are the infinitesimal generators of two uniformly
exponentially stable continuous semigroups. For each $z_{0}\in Z$
the LQR problem (\ref{minimization}-\ref{state equation}) admits a
unique solution $\uu=-S^{-1}B^{\ast }Pz$ where $P\in \mathcal{L}(Z)$
is the unique self-adjoint nonnegative solution of the \textbf{ORE}
\begin{equation}
\left( A^{\ast }P+PA-PBS^{-1}B^{\ast }P+C^{\ast }C\right) z=0
\label{eq.Riccati.opercontinuz}
\end{equation}%
for all $z\in D(A).$ The adjoint $A^{\ast }$ of the unbounded
operator $A$ is defined from $D(A^{\ast })\subset Z$ to $Z$ by the
equality $(A^{\ast
}z,z^{\prime })_{Z}=(z,Az^{\prime })_{Z}$ for all $z\in D(A^{\ast })$ and $%
z^{\prime }\in D(A)$. The adjoint $B^{\ast }\in \mathcal{L}(Z,U)$ of
the bounded
operator $B$ is defined by $(B^{\ast }z,u)_{U}=(z,Bu)_{Z}$, the adjoint $%
C^{\ast }\in \mathcal{L}(Y,Z)$ being defined similarly.

\medskip

Now, we state specific assumptions for the approximation method. Here, $%
\Lambda $ is a given self-adjoint operator on a separable Hilbert
space $X$ which is chosen to be easily approximable on a
semi-decentralized architecture. Generally, $\Lambda $ is chosen
with regard to $A,$ then $\Phi _{Z}$ and $\Phi _{U}$ can be chosen
so that to have also a natural semi-decentralized approximation.

\begin{assumption}[H1]
There exist three integers $n_{Z},$ $n_{U}$ and $n_{Y}\in \mathbb{N}^{\ast }$%
, three isomorphisms $\Phi _{Z}\in \mathcal{L}(X^{n_{Z}},Z),$ $\Phi
_{U}\in \mathcal{L}(X^{n_{U}},U)$ and $\Phi _{Y}\in
\mathcal{L}(X^{n_{Y}},Y)$ and
four matrices of functions $a(\lambda )\in \mathbb{R}^{n_{Z}\times n_{Z}},$ $%
b(\lambda )\in \mathbb{R}^{n_{Z}\times n_{U}}$, $c(\lambda )\in \mathbb{R}%
^{n_{Y}\times n_{Z}}$ and $s(\lambda )\in \mathbb{R}^{n_{U}\times
n_{U}}$ continuous on $I_{\sigma }$ such that
\begin{equation*}
A=\Phi _{Z}a(\Lambda )\Phi _{Z}^{-1}, \quad B=\Phi _{Z}b(\Lambda
)\Phi _{U}^{-1}, \quad C=\Phi _{Y}c(\Lambda )\Phi _{Z}^{-1} \quad
\text{and} \quad S=\Phi _{U}s(\Lambda )\Phi _{U}^{-1}.
\end{equation*}
\end{assumption}

One of the consequences of this assumption, for a system governed by
a partial differential equation posed in a domain $\Omega ,$ is that
both the control and the observation must be distributed throughout
the domain, in conformity with what has been stated from the
beginning$.$

\begin{remark}
$\text{ }$

\begin{enumerate}
\item In case where all operators are function of $\Lambda $, then the
isomorphisms $\Phi $ are or not useful or can be chosen as function of $%
\Lambda $. In both cases $P$ is also a function $p$ of $\Lambda $.

\item Introducing the isomorphisms $\Phi _{Z}$, $\Phi _{Y}$ and $\Phi _{U}$
allows to deal with problems where operators $A$, $B$ and $C$ are
not functions of $\Lambda $.

\item When control is distributed over the entire domain Assumption (H1) is
generally satisfied. In Section \ref{exemple 3}, there is an example
of observation operator $C$ that is not a function of $\Lambda $,
while in the paper \cite{LenYak1} it is the case for $B$ the control
operator.

\item For boundary control or observation problems, it is impossible to find
such isomorphisms. Nevertheless, in Subsection \ref{exemple 4} we
show how to proceed to address some boundary control problems.

\item Multi-scale models with controls at the micro scale, as in \cite%
{LenYak1} and \cite{HuiYak}, are also possible applications.
\end{enumerate}
\end{remark}

\noindent We introduce the \textbf{ARE}
\begin{equation}
a^{T}\left( \lambda \right) p+pa\left( \lambda \right) -pb\left(
\lambda \right) s^{-1}\left( \lambda \right) b^{T}\left( \lambda
\right) p+c^{T}\left( \lambda \right) c\left( \lambda \right)
=0\text{.} \label{eq.Riccati.algebrique}
\end{equation}

\begin{assumption}[H2]
For all $\lambda \in I_{\sigma }$, the \textbf{ARE} (\ref%
{eq.Riccati.algebrique}) admits a unique nonnegative symmetric
solution denoted by $p(\lambda )$.
\end{assumption}

\begin{remark}
\textbf{This assumption is stronger than the typical sufficient
condition for the mere existence of a solution to the Riccati
equation [give ref].}
\end{remark}

We make the following choices for the inner products of $Z$, $U$ and
$Y$:
\begin{equation*}
\left( z,z^{\prime }\right) _{Z}=\left( \Phi _{Z}^{-1}z,\Phi
_{Z}^{-1}z^{\prime }\right) _{X^{n_{Z}}}, ~ \left( u,u^{\prime
}\right) _{U}=\left( \Phi _{U}^{-1}u,\Phi _{U}^{-1}u^{\prime
}\right) _{X^{n_{U}}} ~  \text{and} ~ \left( y,y^{\prime }\right)
_{Y}=\left( \Phi _{Y}^{-1}y,\Phi _{Y}^{-1}y^{\prime }\right)
_{X^{n_{Y}}}.
\end{equation*}
Thus $P$, $Q$ and $p$, $q$ are related as follows.

\begin{theorem}
\label{Th formul P bounded B}If (H1) and (H2) are fulfilled then
\begin{equation*}
P=\Phi _{Z}~p(\Lambda )~\Phi _{Z}^{-1}\quad \text{and}\quad \uu=-Qz
\end{equation*}%
where the controller $Q$ admits the factorization $Q=\Phi
_{U}q(\Lambda
)\Phi _{Z}^{-1}$ with $q(\Lambda )=s^{-1}(\Lambda )b^{T}(\Lambda )p(\Lambda )%
\text{.}$
\end{theorem}

Now, we focus on a semi-decentralized approximation of $Q$ which
reduces to provide such an approximation for $q(\Lambda )$. We
restrict the presentation to the case of bounded operators $\Lambda
$ since they have a bounded spectra. This is sufficient for
applications to systems governed by partial differential equations
in bounded domains.

\begin{assumption}[H3]
The operator $\Lambda $ is bounded and its spectrum $\sigma (\Lambda
)$ is bounded, so there exists $R>0$ with $\sigma (\Lambda )\subset
(-R,R).$
\end{assumption}

This assumption can be relaxed, see Section \ref{extension}.

\begin{assumption}[H4]
The operators $\Phi _{Z}$, $\Phi _{Z}^{-1},$ $\Lambda $ and $(\xi
I-\Lambda )^{-1}$ admit semi-decentralized approximations for all
$\xi \in \mathbb{C}$ with $|\xi |=R$.
\end{assumption}
Now, we introduce two successive approximations $q_{N}(\Lambda )$ and $%
q_{N,M}(\Lambda )$ of $q(\Lambda )$ that play a key role in our
method.

\medskip

\noindent $\vartriangleright $ \textit{The rational approximation }$%
q_{N}(\Lambda )$\textit{:} Since the interval $I_{\sigma }$ is
bounded, each entries $q_{ij}$ of the matrix $q$ admits a rational
approximation on $I_{\sigma }$. This defines a matrix of rational
approximations of $q(\lambda )$,
\begin{equation}
\displaystyle q_{N}\left( \lambda \right) =\frac{\sum%
\limits_{k=0}^{N^{N}}d_{k}\lambda ^{k}}{\sum\limits_{k^{\prime
}=0}^{N^{D}}d_{k^{\prime }}^{\prime }\lambda ^{k^{\prime }}},
\label{apprxfractional}
\end{equation}
to be understood componentwise, so each $d_{k}$, $d_{k^{\prime
}}^{\prime }$ is a matrix and $N=\left( N^{N},N^{D}\right) $ is a
pair of matrices of polynomial degrees. The particular case
$N^{D}=0$ corresponds to a classical polynomial approximation. For
any $\eta >0$ the degrees of approximations
can be chosen so that the uniform estimate%
\begin{equation}
\sup_{\lambda \in I_{\sigma }}\left\vert q\left( \lambda \right)
-q_{N}\left( \lambda \right) \right\vert \leq C_{1}\left( q\right)
\eta \label{estim1}
\end{equation}%
holds.

\medskip

\noindent $\vartriangleright $ \textit{Approximation
}$q_{N,M}(\Lambda )$ \textit{by quadrature of the Cauchy integral:}
For any complex valued
function $g(\theta )$ continuous on $[0,2\pi ],$ we introduce $I_{M}(g)=%
\displaystyle\sum\limits_{\ell =1}^{M}\omega _{\ell }g(\theta _{\ell
})$ a quadrature rule for the integral $I\left( g\right)
=\int_{0}^{2\pi }g\left( \theta \right) ~d\theta $, $(\theta _{\ell
})_{\ell }$ denoting the nodes of a regular subdivision of $[0,2\pi
]$ and $\omega _{\ell }$ the associated quadrature weights. The
quadrature rule is assumed to satisfy an error
estimate as%
\begin{equation}
\left\vert I\left( g\right) -I_{M}\left( g\right) \right\vert \leq
C_{2}\left( g\right) \eta \text{.}  \label{estim2}
\end{equation}%
For $z\in X^{n_{Z}}$ and $\xi =\xi _{1}+i\xi _{2}$ a sufficiently
regular complex contour enlacing $\sigma (\Lambda )$ and not
surrounding any pole of $q_{N}.$ We parameterize it by a parameter
varying in $[0,2\pi ]$. We further introduce the solution
$(v_{i})_{i=1,2}$ of the system
\begin{equation}
\left\{
\begin{array}{c}
\xi _{1}v_{1}-\xi _{2}v_{2}-\Lambda v_{1}=\Re e\left( -i\xi ^{\prime
}q_{N}\left( \xi \right) \right) z, \\
\xi _{2}v_{1}+\xi _{1}v_{2}-\Lambda v_{2}=\Im m\left( -i\xi ^{\prime
}q_{N}\left( \xi \right) \right) z,%
\end{array}%
\right.  \label{transfo}
\end{equation}%
and the second approximation of $q(\Lambda )$ through its realizations%
\begin{equation}
q_{N,M}(\Lambda )z=\frac{1}{2\pi }\sum_{\ell =1}^{M}\omega _{\ell
}v_{1}^{\ell }\text{.}  \label{pNM}
\end{equation}%
We notice that two approximations $p_{N}$ and $p_{N,M}$ of the
function $p$ can be constructed by following the same steps. The
next theorem states the approximations of the operators $P$ and $Q.$

\begin{theorem}
\label{Th approx P bounded B}Under the assumptions (H1-H4), $P$ and
$Q$ can be approximated by one of the two semi-decentralized
approximations
\begin{eqnarray*}
P_{N}=\Phi _{Z}p_{N}\left( \Lambda \right) \Phi _{Z}^{-1} &\text{
and }& \quad Q_{N}=\Phi _{U}q_{N}\left( \Lambda \right) \Phi
_{Z}^{-1} \\
\text{or } P_{N,M}=\Phi _{Z}p_{N,M}\left( \Lambda \right) \Phi
_{Z}^{-1} & \text{ and } & \quad Q_{N,M} =\Phi _{U}q_{N,M}\left(
\Lambda \right) \Phi _{Z}^{-1}\text{.}
\end{eqnarray*}
Moreover, for any $\eta >0,$ there exist $N$ and $M$ such that
\begin{eqnarray*}
\left\| P-P_{N}\right\| _{\mathcal{L}\left( Z\right) } &\leq
&C_{3}\eta \text{, } \left\| Q-Q_{N}\right\| _{\mathcal{L}\left(
Z,U\right) }\leq
C_{3}^{\prime }\eta \\
\text{and }\left\| P-P_{N,M}\right\| _{\mathcal{L}\left( Z\right) }
&\leq &C_{4}\eta \text{, }\left\| Q-Q_{N,M}\right\|
_{\mathcal{L}\left( Z,U\right) }\leq C_{4}^{\prime }\eta ,
\end{eqnarray*}
$C_{3}$, $C_{3}^{\prime }$ and $C_{4},$ $C_{4}^{\prime }$ being
independent of $\eta ,$ $N$ and $M$.
\end{theorem}

\begin{remark}
In the case of a polynomial approximation, i.e. $N^{D}=0$, we can
set a circle as contour $\xi (\theta )=Re^{i\theta }$. For actual
rational approximations, the contour must leave the poles outside,
so we choose an
ellipse centered at $\frac{R_{1}}{2}$ parameterized by $\xi (\theta )=\frac{R_{1}}{2}%
\left( 1+\cos \left( \theta \right) \right) +iR_{2}\sin \left(
\theta
\right) $ where $R_{1}$ and $R_{2}$ are for the major and minor radii and $%
R_{2}$ is small enough.
\end{remark}

\begin{remark}
\label{Rk KadLen}The approximation of $p$ used in \cite{KadLenMr} is
based on Taylor series, so it is applicable only if the interval
$I_{\sigma }$ is sufficiently small. The approximation proposed in
our paper does not suffer from this drawback.
\end{remark}

\begin{remark}
\label{Rk BamPag}In case where the solution $P$ of a Riccati
equation is a kernel operator (see \cite{Lio} for optimal control of
systems governed by
partial differential equations) i.e. $Pz(x)=\int_{\Omega }\overline{p}%
(x,x^{\prime })z(x^{\prime })dx^{\prime }$ and if $\Lambda $ is a
compact
operator then the kernel may be decomposed on a basis of eigenvectors of $%
\Lambda $,%
\begin{equation*}
\overline{p}(x,x^{\prime })=\sum\limits_{k=1}^{\infty }p(\lambda
_{k})\phi _{k}(x)\phi _{k}(x^{\prime }).
\end{equation*}%
The truncation technique used in \cite{BamPag} can be applied to
build a
semi-decentralized approximation of $P$. However, when the decay of $%
\overline{p}$ is not very fast, this technique is not efficient, see
for example the case $p(\lambda )=\lambda $ that may yield from a
LQR problem.
\end{remark}

For concrete real-time computations one can use either of the two formulae (%
\ref{apprxfractional}) or (\ref{pNM}) given that both are
semi-decentralized, but we prefer the second since it does not make
use of powers of $\lambda .$ The reason will become clearer when
discretizing. In a real-time computation, the realization
$q_{N,M}(\Lambda )z$ requires solving $M$ systems (\ref{transfo})
corresponding to $M$ complex values $\xi (\theta _{\ell })$.\ So the
parameter $M$ is essential to evaluate the cost of our algorithm.
The matrix $q_{N}$ is pre-computed off-line once and for all and we
choose $N$ sufficiently large that $q_{N}$ is a very good
approximation of $q$. Consequently, $M$ is the only parameter that
influences the accuracy of the method, except the parameter space
discretization that is discussed now.

\medskip

The end of the section is devoted to spatial discretization. For the
sake of simplicity, the interval is meshed with regularly spaced
nodes separated by a distance $h$.

\medskip

\noindent $\vartriangleright $ \textit{Spatial discretization with
polynomial approximation:} First, we introduce $\Lambda _{h}^{k}$
the finite
differences discretizations of $\Lambda ^{k}$, with $k=1\cdots N$. For $%
N^{D}=0$, the discretization $q_{N,h}$ of $q_{N}$ in
(\ref{apprxfractional}) can be written as
\begin{equation*}
q_{N,h}z_{h}=\sum_{k=0}^{N}d_{k}\Lambda _{h}^{k}z_{h},
\end{equation*}%
where $z_{h}$ is the vector of nodal values of $z$. Their
discretization yields very high errors because the powers of
$\Lambda $. This can be avoided by using the Cauchy formula, i.e.
the equation (\ref{transfo}).

\medskip

\noindent $\vartriangleright $ \textit{Spatial discretization with
Cauchy formula approximation:} For each quadrature point $\xi :=\xi
_{1,\ell }+i\xi _{2,\ell }$, the discrete approximation $\left(
v_{1,h}^{\ell },v_{2,h}^{\ell }\right) $ of $\left( v_{1}^{\ell
},v_{2}^{\ell }\right) $ is the solution of the discrete set of
equations
\begin{equation}
\left\{
\begin{array}{c}
\xi _{1,\ell}v_{1,h}^{\ell }-\xi _{2}v_{2,h}^{\ell }-\Lambda
_{h}v_{1,h}=\Re
e\left( -i\xi_{\ell} q_{N}\left( \xi_{\ell,\ell} \right) \right) z_{h}, \\
\xi _{2,\ell}v_{1,h}^{\ell }+\xi _{1,\ell}v_{2,h}^{\ell }-\Lambda
_{h}v_{2,h}=\Im
m\left( -i\xi_{\ell} q_{N}\left( \xi_{\ell} \right) \right) z_{h}\text{.}%
\end{array}%
\right.  \label{Pbdiscspatiale}
\end{equation}%
Thus we deduce the discretization $q_{N,M,h}$ of the approximation
$q_{N,M}$ in (\ref{pNM}),
\begin{equation}
q_{N,M,h}z_{h}=\frac{1}{2\pi }\sum_{\ell =1}^{M}\omega _{\ell
}v_{1,h}^{\ell }\text{.}  \label{discspatiale}
\end{equation}%
Under the Assumption (H4), we introduce $\Phi _{U,h}$ and $\Phi
_{Z,h}$ the semi-decentralized approximations of $\Phi _{U}$ and
$\Phi _{Z}$. So, the approximations of $\uu_{N}$ and $\uu_{N,M}$ by
a spatial discretization are
\begin{equation}
\uu_{N,h}=-\Phi _{U,h}q_{N,h}\Phi _{Z,h}^{-1}z_{h}\quad \text{and}\quad \uu%
_{N,M,h}=-\Phi _{U,h}q_{N,M,h}\Phi _{Z,h}^{-1}z_{h}.
\label{discspatialeu}
\end{equation}%
This constitutes two different final semi-decentralized approximations of $%
\uu$.

\begin{remark}
\label{remarque10} The approximations $\uu_{N,h}$ and $\uu_{N,M,h}$
are given in the general case where the isomorphisms $\Phi _{Z}$ and
$\Phi _{U}$ are not function of $\Lambda $ only. Therefore,
\textbf{we use our approximation technique to represent} $q(\Lambda
)$. In some cases $\Phi _{Z}$ and $\Phi _{U}$ are function of
$\Lambda $ and then $Q$ is also and the approximation is developed
directly on it that we denote by $k(\Lambda )$,
\begin{equation}
\uu_{N,h}=-k_{N,h}(\Lambda )z_{h}\quad \text{and}\quad \uu%
_{N,M,h}=-k_{N,M,h}(\Lambda )z_{h}\text{.}  \label{discspatialeu1}
\end{equation}%
In the case where $\Phi _{Z}$ and $\Phi _{U}^{-1}Q$ are functions of $%
\Lambda $ then the approximation is developed on $\Phi _{U}^{-1}Q$,
we will
also denote it by $k(\Lambda )$ without risk of confusion,%
\begin{equation*}
\uu_{N,h}=-\Phi _{U,h}k_{N,h}(\Lambda )z_{h}\quad \text{and}\quad \uu%
_{N,M,h}=-\Phi _{U,h}k_{N,M,h}(\Lambda )z_{h}.
\end{equation*}
\end{remark}

\section{Unbounded Control Operators \label{Unbounded control operators}}

When the input operator $B$ is unbounded from $U$ to $Z$ and the
observation
operator $C$ is bounded from $Z$ to $Y$, we use the framework of \cite%
{BanIto} where $V$ is another Hilbert space, $V^{\prime }$ is its
dual space
with respect to the pivot space $Z,$ $A\in \mathcal{L}(V,V^{\prime }),$ $%
B\in \mathcal{L}(U,V^{\prime })$ and $C\in \mathcal{L}(Z,Y)$. A
number of other technical assumptions are not detailed here. The
state equations are written in the sense of $V^{\prime }$ with
$z_{0}\in Z.$ The optimal control is $\uu=-B^{\ast }Pz$ where $P$ is
the unique nonnegative solution of the Riccati operatorial equation
\begin{equation}
\left( A^{\ast }P+PA-PBB^{\ast }P+C^{\ast }C\right) v=0,
\label{eq.Riccatinonborne}
\end{equation}%
for all $v\in V.$ The adjoint $A^{\ast }\in \mathcal{L}(V,V^{\prime
})$ is defined by $\left\langle A^{\ast }v,v^{\prime }\right\rangle
_{V^{\prime
},V}=\left\langle v,Av^{\prime }\right\rangle _{V,V^{\prime }}$ when $%
B^{\ast }\in \mathcal{L}(V^{\prime },U)$ is defined as the adjoint
of a bounded operator. We keep the same inner products for $Z,$ $U$
and $Y$, and those of $V$ and $V^{\prime }$ are
\begin{equation*}
\left( v,v^{\prime }\right) _{V}=\left( \Phi _{V}^{-1}v,\Phi
_{V}^{-1}v^{\prime }\right) _{X^{n_{Z}}}\quad \text{and}\quad \left(
v,v^{\prime }\right) _{V^{\prime }}=\left( \Phi _{V^{\prime
}}^{-1}v,\Phi _{V^{\prime }}^{-1}v^{\prime }\right) _{X^{n_{Z}}}.
\end{equation*}%
Moreover, we choose $J=\Phi _{V}\Phi _{V^{\prime }}^{-1}$ as the
canonical isomorphism from $V^{\prime }$ to $V$ and the duality
product between $V$ and $V^{\prime }$ is
\begin{equation*}
\left\langle v,v^{\prime }\right\rangle _{V,V^{\prime
}}=(v,Jv^{\prime })_{V}.
\end{equation*}

\begin{assumption}[H1']
Same statement as (H1) excepted that
\begin{equation*}
A=\Phi _{V^{\prime }}a(\Lambda )\Phi _{V}^{-1}\quad \text{and}\quad
B=\Phi _{V^{\prime }}b(\Lambda )\Phi _{U}^{-1}
\end{equation*}%
where $\Phi _{V}\in \mathcal{L}(X^{n_{Z}},V)$ and $\Phi _{V^{\prime
}}\in \mathcal{L}(X^{n_{Z}},V^{\prime })$ are two additional
isomorphisms. Moreover,
\begin{equation*}
\Phi _{V}=\phi _{V}(\Lambda )\text{, }\quad \Phi _{Z}=\phi
_{Z}(\Lambda )\quad \text{and}\quad \Phi _{V^{\prime }}=\phi
_{V^{\prime }}(\Lambda )
\end{equation*}%
are some functions of $\Lambda $.
\end{assumption}

Here, the \textbf{ARE} is
\begin{gather}
\phi _{V^{\prime }}\left( \lambda \right) a^{T}\left( \lambda
\right) ~ p ~ \phi _{V^{\prime }}^{-1}\left( \lambda \right) \phi
_{V}\left( \lambda \right) +\phi _{V}\left( \lambda \right) ~ p ~
a\left( \lambda \right)
\notag \\
-\phi _{V}\left( \lambda \right) ~ p ~ b\left( \lambda \right)
s^{-1}\left( \lambda \right) b^{T}\left( \lambda \right) ~ p ~ \phi
_{V^{\prime }}^{-1}\left( \lambda \right) \phi _{V}\left( \lambda
\right)
\label{eq.Riccati.algebriquenonborne} \\
+\phi _{Z}\left( \lambda \right) c^{T}\left( \lambda \right) c\left(
\lambda \right) \phi _{Z}^{-1}\left( \lambda \right) \phi _{V}\left(
\lambda \right) =0\text{.}  \notag
\end{gather}
\begin{assumption}[H2']
For all $\lambda \in I_{\sigma }$, the \textbf{ARE} (\ref%
{eq.Riccati.algebriquenonborne}) admits a unique nonnegative
solution denoted by $p(\lambda )$.
\end{assumption}

\begin{theorem}
\label{Th formul P unbounded B}If (H1',H2') are fulfilled, then
\begin{equation*}
P=\Phi _{V}~p(\Lambda )~\Phi _{V^{\prime }}^{-1}\quad
\text{and}\quad \uu=-Qz
\end{equation*}%
where $Q$ admits the factorization $Q=\Phi _{U}q(\Lambda )$ with
$q(\Lambda )=b^{T}(\Lambda )\phi _{V^{\prime }}^{-1}\phi
_{V}p(\Lambda )\phi _{V^{\prime }}^{-1}.$
\end{theorem}

The following assumptions are necessary for the semi-decentralized
approximation of $P$.

\begin{assumption}[H4']
Same statement that (H4) completed by $\Phi _{V}$, $\Phi _{V^{\prime
}} $ and $\Phi _{V^{\prime }}^{-1}$ admit a semi-decentralized
approximation.
\end{assumption}

In the next statement, the approximations $q_{N}$ and $q_{N,M}$ of
$q$ are built according to the formulae (\ref{apprxfractional}) and
(\ref{pNM}).

\begin{theorem}
\label{Th Approx P unbounded B}Under the Assumptions
(H1',H2',H3,H4'), $P$ and $Q$ can be approximated by one of the two
semi-decentralized approximations
\begin{eqnarray*}
 P_{N}=\Phi _{V}p_{N}(\Lambda )\Phi _{V^{\prime }}^{-1} &
\text{and } & Q_{N}=\Phi _{U}q_{N}(\Lambda ), \\
\text{or } P_{N,M}=\Phi _{V}p_{N,M}(\Lambda )\Phi _{V^{\prime
}}^{-1} & \text{and } & Q_{N,M}=\Phi _{U}q_{N,M}(\Lambda ).
\end{eqnarray*}
Moreover, for any $\eta >0,$ there exist $N$ and $M$ such that
\begin{eqnarray*}
\left\|P-P_{N}\right\|_{\mathcal{L}(V^{\prime },V)} &\leq& C_{3}\eta
, \quad \left\|Q-Q_{N}\right\|_{\mathcal{L}(V^{\prime },U)} \leq
C^{\prime}_{3}\eta , \\
\text{and } \left\|P-P_{N,M}\right\|_{\mathcal{L}(V^{\prime
},V)}&\leq& C_{4}\eta , \quad
\left\|Q-Q_{N,M}\right\|_{\mathcal{L}(V^{\prime },U)} \leq
C^{\prime}_{4}\eta ,
\end{eqnarray*}
$C_{3}$, $C^{\prime}_{3}$, $C_{4}$ and $C^{\prime}_{4}$ being
independent of $\eta ,$ $N$ and $M$.
\end{theorem}

\begin{remark}
An example of unbounded control operators is given in the Subsection \ref%
{exemple 2}.
\end{remark}

The approximations of $\uu$ and $\uu_{h}$ are constructed using the
same method as in the case of bounded control operators.

\section{Extensions\label{Extensions}}

\label{extension} In this section, we mention possible extensions of
the theoretical framework presented above.

The same strategy applies directly to dynamic estimators and
compensators derived by the $H_{2}$ to the $H_{\infty }$ theories.
For instance, the condition $\rho \left( P\overline{P}\right)
<\gamma $ on the spectral radius of the product of the solution of
the two Riccati equation can be expressed under the form of a
condition on the spectral radius of the product of two parameterized
matrices $\rho \left( p\left( \lambda \right) \overline{p}\left(
\lambda
\right) \right) <\gamma $ for all $\lambda \in I_{\sigma }$, see Lemma \ref%
{funct calc generalized} (6).

\medskip

The spectral theory of self-adjoint operators has been chosen for
its relative simplicity. We are aware of its limitation, so we
mention possible extensions based on more general functional calculi
like these developped in \cite{MarSan} or \cite{Haa} to cite only
two.

\medskip

Other frameworks for the well-posedness of the LQR problem can be
used. In particular, this of \cite{Lasiecka} for optimal control
with unbounded observation and control may be incorporated in this
approach.

\section{Proofs\label{Proofs}}

First, we remark that for $E$ and $F$ two Hilbert spaces and $\Phi $
an
isomorphism from $E$ to $F,$ if $F$ is equipped with the inner product $%
(z,z^{\prime })_{F}=(\Phi ^{-1}z,\Phi ^{-1}z^{\prime })_{E}$ then
$\Phi ^{\ast }=\Phi ^{-1}$. In the next lemma, we state few
functional calculus properties.

\begin{lemma}
\label{funct calc}For $\Lambda $ a self-adjoint operator on a
separable Hilbert space $X$, and for $f$, $g$ two functions
continuous on $I_{\sigma }$

\begin{enumerate}
\item $f(\Lambda )$ is self-adjoint;

\item for $\mu \in \mathbb{R}$, $(\mu f)(\Lambda )=\mu f(\Lambda )$ on $%
D(f(\Lambda ))$;

\item $(f+g)(\Lambda )=f(\Lambda )+g(\Lambda )$ on $D(f(\Lambda ))\cap
D(g(\Lambda ))$;

\item $g(\Lambda )f(\Lambda )=(g$ $f)(\Lambda )$ when the range of $%
f(\Lambda )$ is included in $D(g(\Lambda ))$;

\item if $f\neq 0$ in $I_{\sigma }$ then $f(\Lambda )^{-1}$ exists and is
equal to $\frac{1}{f}(\Lambda )$;

\item if $f(\lambda )\geq 0$ for all $\lambda \in I_{\sigma }$ then $%
f(\Lambda )\geq 0$;

\item $||f(\Lambda )x||_{X}^{2}$ $\leq $ $\sup\limits_{\lambda \in I_{\sigma
}}|f(\lambda )|^{2}||x||_{X}^{2}$ for all $x\in D(f(\Lambda ))$.
\end{enumerate}
\end{lemma}

\begin{proof}
The proofs of the first five statements can be found in
\cite{DauLio}. We
prove (\emph{6}) i.e. that $\displaystyle\sum\limits_{i,j=1}^{n}(f_{ij}(%
\Lambda )z_{j},z_{i})_{X}\geq 0$. First, assume that $I_{\sigma }$
is bounded. We recall that for a function $g$ continuous on
$I_{\sigma }$ and for $z\in X,$ the integral $\int_{\sigma _{\min
}}^{\sigma _{\max }}g(\lambda )dE_{\lambda }z$ is defined as the
strong limit in $X$ of the
Riemann sums, see \cite{DauLio}, $\displaystyle\sum\limits_{k=1}^{p}g(%
\lambda _{k}^{\prime })(E_{\lambda _{k+1}}-E_{\lambda _{k}})z$ when $%
\max\limits_{k}|\lambda _{k+1}-\lambda _{k}|$ vanishes, where
$\lambda _{k}^{\prime }\in \lbrack \lambda _{k},\lambda _{k+1}]$ and
$\sigma _{\min
}=\lambda _{1}<\lambda _{2}...<\lambda _{p}=\sigma _{\max }$. When $%
I_{\sigma }$ is not bounded, we use a subdivision of a bounded interval $I_{%
\widetilde{\sigma }}=(\widetilde{\sigma }_{\min },\widetilde{\sigma
}_{\max })$ and the integral $\int_{\sigma _{\min }}^{\sigma _{\max
}}g(\lambda )dE_{\lambda }z$ is defined by passing to the limit in
the integral bounds.
Let us establish that the Riemann sum $\displaystyle\sum\limits_{i,j=1}^{n}%
\sum\limits_{k=1}^{p}f_{ij}(\lambda _{k}^{\prime })((E_{\lambda
_{k+1}}z_{j},z_{i})-(E_{\lambda _{k}}z_{j},z_{i}))$ is nonnegative,
so the result will follow by passing to the limit. Since
$(E_{\lambda _{k+1}}z_{j},z_{i})-(E_{\lambda
_{k}}z_{j},z_{i})=((E_{\lambda
_{k+1}}-E_{\lambda _{k}})z_{j},z_{i})=(y_{j}^{k},y_{i}^{k})$ where $%
y_{j}^{k}=(E_{\lambda _{k+1}}-E_{\lambda _{k}})z_{j},$ then the
Riemann sum
is the sum over $k$ of the nonnegative terms $\displaystyle%
\sum\limits_{i,j=1}^{n}f_{ij}(\lambda _{k}^{\prime
})(y_{j}^{k},y_{i}^{k})$ which in turn is nonnegative.

Now we prove (\emph{7}):
\begin{eqnarray*}
||f(\Lambda ))x||_{X}^{2} =\int_{\sigma _{\min }}^{\sigma _{\max
}}|f(\lambda )|^{2}\text{ }d||E_{\lambda }x||_{X}^{2} &\leq
&\sup_{\lambda \in I_{\sigma }}|f(\lambda )|^{2}\int_{\sigma _{\min
}}^{\sigma _{\max
}}d||E_{\lambda }x||_{X}^{2} \\
&\leq &\sup_{\lambda \in I_{\sigma }}|f(\lambda )|^{2}||x||_{X}^{2}.
\end{eqnarray*}
\end{proof}

For two integers $n_{E}$, $n_{F}$, a $n_{E}\times n_{F}$ matrix $f$
of functions continuous on $I_{\sigma }$ and two Hilbert spaces $E$,
$F$ isomorphic with $X^{n_{E}}$ and $X^{n_{F}}$ by $\Phi _{E}^{-1}$
and $\Phi _{F}^{-1}$ respectively, we introduce the so-called
generalized matrix of functions of $\Lambda $: $f^{\phi }(\Lambda
)=\Phi _{E}f(\Lambda )\Phi _{F}^{-1}\in \mathcal{L}(F,E)$ with
domain $D(f^{\phi }(\Lambda ))=\Phi _{F}D(f(\Lambda ))$. For the
sake of shortness, the spaces $E$ and $F$ do not appear explicitly
in the notation $f^{\phi }$, so they will be associated to each
matrix at the beginning of their use. Then, no confusion will be
possible. In the next lemma, we state some properties of generalized
matrices of functions.

\begin{lemma}
\label{funct calc generalized}For any generalized matrices of functions of $%
\Lambda ,$ $f^{\phi }(\Lambda )=\Phi _{E}f(\Lambda )\Phi _{F}^{-1}$ and $%
g^{\phi }(\Lambda )=\Phi _{E}g(\Lambda )\Phi _{F}^{-1}$, and any
real number $\mu $,

\begin{enumerate}
\item $(f^{\phi }(\Lambda ))^{\ast }=(f^{T})^{\phi }(\Lambda )$;

\item $\mu f^{\phi }(\Lambda )=(\mu f)^{\phi }(\Lambda )$ on $D(f^{\phi
}(\Lambda ))$;

\item $f^{\phi }(\Lambda )+g^{\phi }(\Lambda )=(f+g)^{\phi }(\Lambda )$ on $%
D(f^{\phi }(\Lambda ))\cap D(g^{\phi }(\Lambda ))$;

\item for another Hilbert space $G$ and $g^{\phi }(\Lambda )=\Phi
_{F}g(\Lambda )\Phi _{G}^{-1}$, $f^{\phi }(\Lambda )g^{\phi
}(\Lambda )=(fg)^{\phi }(\Lambda )=\Phi _{E}(fg)(\Lambda )\Phi
_{G}^{-1}$ when the range $R(f^{\phi }(\Lambda ))\subset D(g^{\phi
}(\Lambda ))$;

\item when $F=E,$ if $f(\lambda )\geq 0$ for all $\lambda \in I_{\sigma }$
then $f^{\phi }(\Lambda )\geq 0$;

\item $\sigma (f^{\phi }(\Lambda ))=\sigma (f)$.
\end{enumerate}
\end{lemma}

\begin{proof}
The properties (\emph{1-4}) are direct consequences of Lemma \ref{funct calc}%
. For the derivation of (\emph{5}) we remark that for $z\in
D(f^{\phi}(\Lambda ))\subset E\text{, }(f^{\phi }(\Lambda
)z,z)_{E}=(f(\Lambda )\Phi _{E}^{-1}z,\Phi _{E}^{-1}z)_{X^{n_{E}}}$
which is nonnegative if $f(\Lambda ) $ is nonnegative. The
conclusion uses Lemma \ref{funct calc} (\emph{5}). Finally, the
derivation of (\emph{6}) is a direct consequence of the definition
of the spectrum of an operator.
\end{proof}

\begin{proof}[Proof of Theorem \protect\ref{Th formul P bounded B}]
From Lemma \ref{funct calc generalized} (\emph{1}) and (\emph{4}),
\begin{equation*}
A^{\ast }=\Phi _{Z}a^{T}(\Lambda )\Phi _{Z}^{-1},\quad BB^{\ast
}=\Phi _{Z}b(\Lambda )b^{\ast }(\Lambda )\Phi _{Z}^{-1}\quad
\text{and} \quad C^{\ast }C=\Phi _{Z}c^{\ast }(\Lambda )c(\Lambda
)\Phi _{Z}^{-1}
\end{equation*}%
are some generalized matrices of functions of $\Lambda $ on $Z$. We
write
\begin{equation*}
e(\lambda )=a^{T}(\lambda )~p(\lambda )+p(\lambda )~a(\lambda
)-p(\lambda )~b(\lambda )b^{T}(\lambda )~p(\lambda )+c^{T}(\lambda
)c(\lambda ),
\end{equation*}%
so by construction $e(\lambda )=0$ and $e(\Lambda )=0.$ Multiplying
the last equality by $\Phi _{Z}$ to the left and by $\Phi _{Z}^{-1}$
to the right, using Lemma \ref{funct calc generalized} (\emph{3})
and (\emph{4}), and
posing $\widetilde{P}=\Phi _{Z}~p(\Lambda )~\Phi _{Z}^{-1}$ we find that $%
\widetilde{P}$ satisfies the Riccati equation (\ref{eq.Riccati.opercontinuz}%
). Next, the nonnegativity and symmetry of $p$ with Lemma \ref{funct
calc generalized} (\emph{1}) and (\emph{5}) yield the nonnegativity
and
self-adjointness of $\widetilde{P}$. Finally, we conclude that $P=\widetilde{%
P}$ thanks to uniqueness of the solution, so $\uu=-Qz$ where
$Q=S^{-1}B^{\ast }\widetilde{P}=\Phi _{U}s^{-1}(\Lambda
)b^{T}(\Lambda )p(\Lambda )\Phi _{Z}^{-1}$.
\end{proof}

\begin{proof}[Proof of Theorem \protect\ref{Th approx P bounded B}]
The estimate $||q(\Lambda )-q_{N}(\Lambda )||_{\mathcal{L}\left(
X^{n_{Z}},X^{n_{U}}\right) }$ results from (\ref{estim1}) and Lemma \ref%
{funct calc} (\emph{7}). In the following, we derive the estimate
\begin{equation*}
||q_{N}(\Lambda )-q_{N,M}(\Lambda )||_{\mathcal{L}\left(
X^{n_{Z}},X^{n_{U}}\right) }\leq C_{5}\eta .
\end{equation*}%
Since $q_{N}$ is holomorphic in $\mathbb{C}$ and $\Lambda $ is a
bounded operator on $X$ with a spectrum included in $(-R,R)$,
$p_{N}(\Lambda )$ may be represented by the Cauchy formula, see
\cite{Yos},
\begin{equation*}
q_{N}(\Lambda )=\frac{1}{2i\pi }\int_{\mathcal{C}(R)}q_{N}(\xi )(\xi
I-\Lambda )^{-1}d\xi
\end{equation*}%
where $\mathcal{C}(R)\subset \mathbb{C}$, provided that all its
poles are out of the contour $\mathcal{C}(R)$. By choosing $\xi $,
function of $\theta $, with $\theta \in (0,2\pi )$ as a
parametrization of $\mathcal{C}(R)$, we find
\begin{equation*}
q_{N}(\Lambda )=\frac{1}{2\pi }\int_{0}^{2\pi }-i\xi ^{\prime
}q_{N}(\xi )(\xi I-\Lambda )^{-1}d\theta .
\end{equation*}%
Then, we use the quadrature formula to approximate $q_{N}(\lambda )$
by
\begin{equation*}
q_{N,M}(\lambda )=\frac{1}{2\pi }I_{M}(-i\xi ^{\prime }q_{N}(\xi
)(\xi -\lambda )^{-1}).
\end{equation*}%
Combining the estimate (\ref{estim2}) and Lemma \ref{funct calc} (5)
yields the wanted estimate. The triangular inequality yields
\begin{eqnarray*}
\left\Vert q(\Lambda )-q_{N,M}(\Lambda )\right\Vert
_{\mathcal{L}\left( X^{n_{Z}},X^{n_{U}}\right) } &\leq &\left\Vert
q(\Lambda )-q_{N}(\Lambda
)\right\Vert _{\mathcal{L}\left( X^{n_{Z}},X^{n_{U}}\right) } +\left\Vert q_{N}(\Lambda )-q_{N,M}(\Lambda )\right\Vert _{\mathcal{L}%
\left( X^{n_{Z}},X^{n_{U}}\right) } \\
&\leq &(C_{3}+C_{5})\eta = C_{4}\eta
\end{eqnarray*}%
with $C_{4}=C_{3}+C_{5}$. Consequently, the expression (\ref{pNM}) of $%
q_{N,M}(\Lambda )z$ is obtained by posing $v^{\ell }=-i\xi _{\ell
}^{\prime }q_{N}(\xi _{\ell })(\xi _{\ell }-\Lambda )^{-1}z$.
\end{proof}

\begin{remark}
\label{Rq ellipse}The implementation of the Cauchy integral formula
requires that the function is holomorphic inside the contour. In the
case of an unknown function like the function $q$, it is generally
difficult to determine its domain of holomorphy, so it is easier to
use a rational approximation $q_{N}$ whose poles are under control.
\end{remark}

\begin{proof}[Proof of Theorem \protect\ref{Th formul P unbounded B}]
The derivation of the expression $A^{\ast }=J^{-1}\Phi _{V}a^{\ast
}(\Lambda
)\Phi _{V^{\prime }}^{-1}J^{-1}$ is straightforward provided that $%
\left\langle u,v\right\rangle _{V^{\prime
},V}=(Ju,v)_{V}=(u,J^{-1}v)_{V^{\prime }}$. Since $J=\Phi _{V}\Phi
_{V^{\prime }}^{-1}$ this expression is simplified as $A^{\ast
}=\Phi _{V^{\prime }}a^{\ast }\Phi _{V}^{-1}$. Then,
(\ref{eq.Riccatinonborne}) is equivalent to
\begin{gather*}
\left[ \phi _{V^{\prime }}(\Lambda )a^{\ast }(\Lambda )P\phi
_{V^{\prime }}^{-1}(\Lambda )\phi _{V}(\Lambda )+\phi _{V}(\Lambda
)Pa(\Lambda )\right.
\\
-\phi _{V}(\Lambda )Pb(\Lambda )b^{\ast }(\Lambda )P\phi _{V^{\prime
}}^{-1}(\Lambda )\phi _{V}(\Lambda ) \\
+\left. \phi _{Z}c^{\ast }(\Lambda )c(\Lambda )\phi _{Z}^{-1}\phi
_{V}\right] x=0.
\end{gather*}
Finally, the complete proof follows the same steps as in Theorem
\ref{Th formul P bounded B}.
\end{proof}

The proof of Theorem \ref{Th Approx P unbounded B} is similar to the
one of Theorem \ref{Th approx P bounded B}.

\section{Applications and Numerical Results\label{Applications}}

\label{Application}

We present four applications to illustrate different aspects of the
theory. In Examples 1, 3 and 4, the input operator $B$ is bounded
when in Example 2 it is not. Then, we consider cases where the
operators $B$ and $C$ are functions of $\Lambda $ (Examples 1, 2 and
4), and a case where it is not (Example 3). Most examples are
devoted to internal control, nevertheless through the example of
Subsection \ref{exemple 4} it is shown how to tackle a boundary
control problem. In almost all cases, efficient algorithms are
described. The presentation of the examples 1, 3 and 4 follows the
same plan with three sub-Sections. The first one includes the state
equation, the functional to be minimized and some semi-decentralized
controls resulting of our approach. Their derivation is detailed in
the second sub-Section. As for the third, it discusses numerical
results.

The functional analysis is carried out in Sobolev spaces defined for
any
integer $k\in \mathbb{N}^{\ast }$ and any domain $\Omega \subset \mathbb{R}%
^{d}$ by
\begin{eqnarray*}
H^{k}(\Omega ) &=&\{v\in L^{2}(\Omega )\quad |\quad \nabla ^{j}v\in
L^{2}(\Omega )^{d^{j}}\quad \text{for all}\quad 1\leq j\leq k\} \\
\text{and }H_{0}^{k}(\Omega ) &=&\{v\in H^{k}(\Omega )\quad |\quad
\nabla ^{j}v=0\text{ on }\partial \Omega \quad \text{for all}\quad
0\leq j\leq k-1\}.
\end{eqnarray*}%
The boundary $\partial \Omega$ of $\Omega$ is always assumed to be
sufficiently regular to avoid any singularity and thus to simplify
the
choice of the isomorphisms $\Phi $. Its outward unit normal is denoted by $%
\nu $. For $N\in \mathbb{N}$, $\mathbb{P}_{N}$ represents the set of $N^{%
\text{th}}$-order polynomials.

\subsection{Example 1: The heat equation with a bounded control operator}

\label{exemple 1}

In this example, observation and control operators are bounded.

\subsubsection{The state equation and a choice of semi-decentralized
controllers}

Consider a system modeled by the heat equation posed in a domain
$\Omega \subset \mathbb{R}^{d}$, with homogeneous Dirichlet boundary
conditions. The control is distributed over the whole domain, so the
state $z:=w$ is solution to the boundary value problem,
\begin{equation}
\left\{ \begin{aligned} \partial _{t} w(t,x) & = \Delta w(t,x) +
\beta u(t,x) && \text{in } \mathbb{R}^{+\ast }\times \Omega \text{,}
\\ w(t,x) & = 0 && \text{on } \mathbb{R}^{+\ast }\times \partial
\Omega , \\ w(0,x) & = w_{0} && \text{in }\Omega ,
\end{aligned}\right.  \label{exemple1}
\end{equation}%
and the functional $\mathcal{J}\left( w_{0};u\right)
=\int_{0}^{+\infty }\left\Vert w\right\Vert _{L^{2}\left( \Omega
\right) }^{2}+\left\Vert \gamma u\right\Vert _{L^{2}\left( \Omega
\right) }^{2}dt$ is to be minimized. Here, $\beta $ and $\gamma $
are two nonnegative continuous
functions in $\Omega $. We apply the theory with the self-adjoint operator $%
\Lambda =\left( -\Delta \right) ^{-1}$, defined as the inverse of
the Laplace operator $-\Delta :H_{0}^{1}(\Omega )\cap H^{2}(\Omega
)\rightarrow L^{2}(\Omega )$.

\noindent $\vartriangleright $ \textit{Linear approximation:} The
approximation (\ref{polynomial approximation}) with a first-degree
polynomial yields%
\begin{equation*}
\uu_{1}=-\frac{\beta }{\sqrt{\gamma }}(d_{0}+d_{1}\Lambda )w,
\end{equation*}%
so in the special case $\gamma =\beta =1$, $\uu_{1}$ is the solution
to the boundary value problem
\begin{equation*}
-\Delta \uu_{1}=d_{0}\Delta w-d_{1}w\text{ in }\Omega, \quad \text{with}~\uu%
_{1}=w=0\text{ on }\partial \Omega .
\end{equation*}%
In the one-dimensional case $\Omega =]0,\pi \lbrack ,$ we apply
Algorithm 1 described hereafter to find $d_{0}=2.23\times 10^{-2}$
and $d_{1}=0.407$. Such $\uu_{1}$ constitutes a semi-decentralized
control before spatial discretization. The Laplace operator i.e. the
second order derivative may be approximated by a three-point
centered finite difference scheme, with
solution $(\uu_{1,j})_{j=0,\ldots,\mathcal{N}}$ that approximates the solutions $%
\uu_{1}(x_{j})$ at the $(\mathcal{N}+1)$ nodes of a subdivision $%
(x_{j}=jh)_{j=0,\ldots,\mathcal{N}}$ with $h=\frac{\pi }{\mathcal{N}},$%
\begin{equation*}
-\left( \uu_{1,j-1}-2\uu_{1,j}+\uu_{1,j+1}\right) =d_{0}\left(
w_{j-1}-2w_{j}+w_{j+1}\right) -d_{1}h^{2}w_{j},\quad j=1,\ldots,\mathcal{N%
}-1,
\end{equation*}%
completed by the boundary conditions
$\uu_{1,0}=\uu_{1,\mathcal{N}}=0$. Here
$w_{j}=w(x_{j})$ for $j=0,\ldots,\mathcal{N}$ that satisfy $w_{0}=w_{\mathcal{N}%
}=0$. After elimination of $\uu_{1,0}$ and $\uu_{1,\mathcal{N}},$
the scheme can be written in matrix form,
\begin{equation}
\lbrack -\Delta _{h}]\uu_{1,h}=-d_{0}[-\Delta _{h}]w_{h}-d_{1}w_{h},
\label{uNh'}
\end{equation}
where $[-\Delta _{h}]=\dfrac{1}{h^{2}}%
\begin{pmatrix}
2 & -1 &  &  &  \\
-1 & 2 & \ddots &  &  \\
& \ddots & \ddots & \ddots &  \\
&  & \ddots & 2 & -1 \\
&  &  & -1 & 2%
\end{pmatrix}%
$, $\uu_{1,h}=
\begin{bmatrix}
\uu_{1,1} \\ \vdots \\ \uu_{1,\mathcal{N}-1}
\end{bmatrix}
$ and $w_{h}=
\begin{bmatrix}
w_{1} \\ \vdots \\ w_{\mathcal{N}-1}
\end{bmatrix}
$. This is the fully discretized problem of the semi-decentralized
control approximated by a linear polynomial.

\medskip

\noindent $\vartriangleright $ \textit{Approximation through the
Cauchy formula combined with a polynomial approximation:} To build
the approximated optimal control,
\begin{equation}
\uu_{N,M,h}=-\frac{1}{2\pi }\sum_{\ell =1}^{M}\omega _{\ell }v_{1,h}^{\ell }%
,  \label{app2ex1}
\end{equation}%
the approximation $v_{i,h}^{\ell }$ of $%
\begin{bmatrix}
v_{i}^{\ell }(x_{1}) & \ldots & v_{i}^{\ell }(x_{\mathcal{N}-1})%
\end{bmatrix}%
^{T}$ is computed by solving the system (\ref{Pbdiscspatiale}), that
we
rewrite in the matrix form,%
\begin{equation}
\begin{bmatrix}
\xi _{1}-[-\Delta _{h}]^{-1} & -\xi _{2} \\
\xi _{2} & \xi _{1}-[-\Delta _{h}]^{-1}%
\end{bmatrix}%
\begin{bmatrix}
v_{1,h}^{\ell } \\
v_{2,h}^{\ell }%
\end{bmatrix}%
=%
\begin{bmatrix}
\Re e\left( -i\xi ^{\prime }p_{N}\left( \xi \right) \right) w_{h} \\
\Im m\left( -i\xi ^{\prime }p_{N}\left( \xi \right) \right) w_{h}%
\end{bmatrix}
\label{vh1d}
\end{equation}%
for each quadrature point $\xi :=\xi _{1,\ell }+i\xi _{2,\ell }$, where $%
p_{N}(\lambda )$ is a polynomial approximation of $p(\lambda )$.%

\subsubsection{Construction of the semi-decentralized controllers}

We detail the derivation of the polynomial approximation
$p_{N}(\Lambda )w$ of $Pw$ required both for the linear
approximation and in (\ref{vh1d}). We set $U=Z=L^{2}\left( \Omega
\right) $ thus $A=\Delta $ is an isomorphism from its domain
$D\left( \Delta \right) =H^{2}\left( \Omega \right) \cap
H_{0}^{1}\left( \Omega \right) $ into $Z$, see \cite{Gri}. Furthermore, $%
Y=U=Z$ and $B=C=S=I$. We set $X=Z$, $\Lambda =\left[ -\Delta \right]
^{-1}$ which is compact, so it has a bounded positive spectrum with
an accumulation point at zero (but $0\not\in \sigma (\Lambda )$),
see \cite{SanSan}. Thus we can choose $\Phi _{Z}=\Phi _{Y}=\Phi
_{U}=I.$ Moreover, when $\beta =\gamma
=1$ the coefficients $a\left( \lambda \right) =-\frac{1}{\lambda }$, $%
b=c=s=1 $ are continuous on $I_{\sigma }=\left( 0,\sigma _{\max
}\right] $ and the \textbf{ARE} reads
\begin{equation}
p^{2}\left( \lambda \right) +\frac{2}{\lambda }p\left( \lambda
\right) -1=0. \label{ex1.numricc}
\end{equation}%
Its exact nonnegative solution, established only to the calculations
of errors, is%
\begin{equation}
p(\lambda )=\frac{-1+\sqrt{1+\lambda ^{2}}}{\lambda }.
\label{ex1.solricc}
\end{equation}%
We observe that $p(\lambda )$ is sufficiently regular to be
accurately
approximated by polynomials in $I_{\sigma }$. The \textbf{ARE} (\ref%
{ex1.numricc}) is equivalent to the weak formulation
\begin{equation}
\int_{I_{\sigma }}\left( \lambda p^{2}+2p-\lambda \right) \eta
\left( \lambda \right) d\lambda =0\text{ for all }\eta \in
\mathcal{C}^{0}\left( I_{\sigma }\right)  \label{algo1 riccati}
\end{equation}%
to which we apply the spectral method with Legendre polynomials (see \cite%
{BerMad1992} for instance) to find the equation satisfied by the
polynomial approximation $p_{N}$. The computation of the integral is
done exactly by using the Legendre-Gauss-Lobatto (\textbf{LGL})
quadrature formula analyzed in \cite{BerMad1991}, \cite{CroMig} and
\cite{DavRab}. The resolution of the nonlinear problem is achieved
by the iterative semi-implicit scheme described below, where
$\varepsilon $ is the stop criteria.
\begin{algorithm}[!h]
\caption{Semi-Implicit scheme applied to (\ref{algo1 riccati})}
\begin{algorithmic} [1]

\STATE $p_N^{0}$ given

\STATE $(m+1)^{\text{th}}$ step : knowing $p_N^{m}\in
\mathbb{P}_{N}$, find $p_N^{m+1}\in \mathbb{P}_{N}$ such that
\begin{equation*}
\int_{I_{\sigma }} p_N^{m+1}\left( \lambda \right)\left( \lambda
p_N^{m}\left( \lambda \right) +2 \right)\eta \left( \lambda \right)
d\lambda = \int_{I_{\sigma }}\lambda \eta \left( \lambda \right)
d\lambda,
\quad \forall \eta \in \mathbb{P}_{N}. 
\end{equation*}

\STATE If $\|p_N^{m+1}-p_N^{m}\|_{L^2\left(
I_{\sigma}\right)}\leq\varepsilon$ then terminate the algorithm else
return to Step 2.
\end{algorithmic}
\end{algorithm}


\subsubsection{Numerical results}

We analyze separately the three sources of discretization error: the
error of approximation of $p$ by a polynomial $p_{N}$, the error in
the quadrature of the Cauchy formula and the spatial discretization
error. We also discuss the convergence of Algorithm 1.

\noindent $\vartriangleright $ \textit{Polynomial approximation:}
The difference between successive iterations $\Vert
p_{N}^{m+1}-p_{N}^{m}\Vert _{L^{2}(I_{\sigma })}$ of Algorithm 1
decays exponentially. For $N=10$ and for the initial solution
$p_{N}^{0}=0$, the exponential decay rate is equal to $-1.80$. Let
us denote by $p_{N}$ the polynomial obtained after convergence of
$p_{N}^{m}$ by Algorithm 1, the convergence error $\Vert
p_{N}^{m}-p_{N}\Vert _{L^{2}(I_{\sigma })}$ is also exponentially
decaying with an exponential decay rate of $-1.82$. In addition, as
it is usual for spectral methods, the relative error
\begin{equation*}
e=\frac{\left\Vert p-p_{N}\right\Vert _{L^{2}\left( I_{\sigma }\right) }}{%
\left\Vert p\right\Vert _{L^{2}\left( I_{\sigma }\right) }}
\end{equation*}
of the polynomial approximation decreases exponentially with $N$.
Here, the exponential decay rate is $-1.61$.

\medskip

\noindent $\vartriangleright $\textit{\ Approximation through the
Cauchy formula combined with a polynomial approximation:} Because of
the absence of poles in $p_{N}$, the choice of the contour of the
Cauchy formula is free of constraints as long as it surrounds
$I_{\sigma }$. We have chosen a circle parameterized by $\xi (\theta
)=Re^{i\theta }$, with $\theta \in \lbrack 0,2\pi ]$. Then, we have
set the polynomial degree sufficiently large so that the error $e$
can be neglected. The numerical integrations have been performed
with a standard trapezoidal quadrature rule. Figure \ref{choixM1}
represents the relative error
\begin{equation*}
E=\frac{\left\Vert p-p_{N,M}\right\Vert _{L^{2}\left( I_{\sigma }\right) }}{%
\left\Vert p\right\Vert _{L^{2}\left( I_{\sigma }\right) }}
\end{equation*}
between $p$ and $p_{N,M}$ for various values of the radius $R$. It
converges exponentially with respect to $M$ towards $e,$ and the
exponential decay rate is a decreasing function of $R$.

\begin{center}
\begin{figure}[h]
\begin{center}
{\normalsize \scalebox{0.4}{\includegraphics*{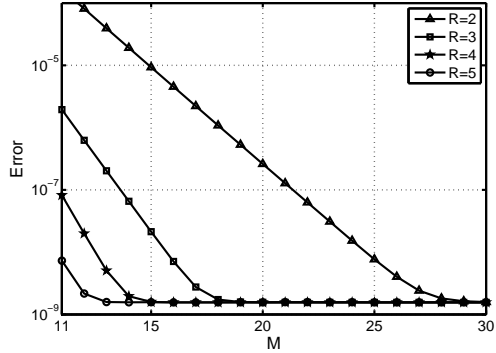}} }
\end{center}
\caption{Error $E$ in logarithmic scale as a function of $M$ for
different values of $R$ and for $N=10$} \label{choixM1}
\end{figure}
\end{center}

\noindent $\vartriangleright $\textit{\ Spatial discretization:}
Computations have been carried out for $\uu_{N,h}$ defined in
(\ref{uNh'}) with $N=1$ and for $\uu_{N,M,h}$ defined in
(\ref{app2ex1}) with $R=5$, $N=10$ and $M=11$ so that $e$ is in the
range of $10^{-9}$ and is negligible compared to $E$. The
approximation (\ref{app2ex1}) is obtained from the formula
(\ref{discspatialeu}) by substituting $\Phi _{U}$ and $\Phi _{Z}$ by
the identity operator and by using the centered finite difference
scheme of the second order derivative, i.e. by replacing $\Lambda $
by its discretization $\Lambda _{h}$. The spatial discretizations
are compared to the expression of the approximations $\uu_{N}(t,x)$
and $\uu_{N,M}(t,x)$ that we calculate thanks to the modal
decomposition of the operator $\partial _{xx}^{2}$ with homogeneous
Dirichlet boundary conditions. It
comes%
\begin{equation*}
\uu_{N}(t,x)=-\sum\limits_{i\in \mathbb{N}^{\ast }}w_{i}e^{-\left(
\lambda _{i}^{-1}+p_{N} \left( \lambda _{i}\right) \right)
t}p_{N}\left( \lambda _{i}\right) \phi _{i}\left( x\right) ,
\end{equation*}%
where $\lambda _{i},$ $\phi _{i}$ and $w_{i}$ represent respectively the $i^{%
\text{th}}$ eigenvalue, the $i^{\text{th}}$ eigenvector and the $i^{\text{th}%
}$ modal coefficient of the initial condition. The same expression
holds for $\uu_{N,M}(t,x)$ after replacement of $p_{N}$ by
$p_{N,M}$. Then, the errors,
\begin{equation*}
\displaystyle\frac{\int_{0}^{T}\left(\sum_{i=0}^{\mathcal{N}}\left| \uu%
_{N,i}-\uu_{N}(x_{i})\right|^{2}\right)^{\frac{1}{2}}dt}{\int_{0}^{T}\left(\sum_{i=0}^{%
\mathcal{N}}\left\vert \uu_{N,i}\right\vert ^{2}\right)^{\frac{1}{2}}dt} \quad \text{and} \quad \frac{%
\int_{0}^{T}\left(\sum_{i=0}^{\mathcal{N}}\left| \uu_{N,M,i}-\uu%
_{N,M}(x_{i})\right| ^{2}\right)^{\frac{1}{2}}dt}{\int_{0}^{T}\left(\sum_{i=0}^{\mathcal{N}%
}\left| \uu_{N,M}(x_{i})\right| ^{2}\right)^{\frac{1}{2}}dt},
\end{equation*}%
are known to be theoretically quadratic with respect to $h$ the
spatial discretization step, which is confirmed by our experiments.

\subsection{Example 2: Heat equation with unbounded control operator}

\label{exemple 2}

In this example, the control operator is internal and unbounded and
the
observation operator is internal and bounded. We apply the theory of Section %
\ref{Unbounded control operators} without going into much detail as
for other examples. We only describe the state equation and the
functional analysis framework.

\subsubsection{The state equation}

We keep the heat equation as the state equation with the same control space $%
U \subset L^{2}(\Omega )$ and the same functional $\mathcal{J}$ but
the control operator is replaced by an unbounded one defined in the
distribution sense by $\left\langle Bu,v\right\rangle =-\beta
\int_{\Omega }u\beta _{1}.\nabla vdx$, where $\beta _{1}$ is a
vector of $\mathbb{R}^{d}.$

\subsubsection{The functional framework}

First, we pose $V=H_{0}^{1}(\Omega )$, so $A=\Delta $ is an
isomorphism from $V$ into $V^{\prime }$ from which we define
$J=(-A)^{-1}$. It allows to give a precise definition of $B$: for
all $v\in V$, $\left\langle Bu,v\right\rangle _{V^{\prime
},V}=-\int_{\Omega }u\beta _{1}.\nabla v~dx$ for $u\in U$. Let us
compute $B^{\ast }$ defined by $\left( Bu,v\right) _{V^{\prime
}}=\left( u,B^{\ast }v\right) _{L^{2}(\Omega )}$ for $u\in U$ and
$v\in V^{\prime }$. Since $\left( Bu,v\right) _{V^{\prime
}}=\left\langle Bu,Jv\right\rangle _{V^{\prime },V}=-\left( u,\beta
_{1}.\nabla Jv\right) _{L^{2}(\Omega )}$ then $B^{\ast }v=-\beta
_{1}\nabla Jv$. We introduce the kernel of $B$, $K_{B}=\{u\in
L^{2}\left( \Omega \right) |u$ constant in the direction$~\beta
_{1}\},$ $U=L^{2}(\Omega )/K_{B} $ and the kernel of $B^{\ast }$,
$K_{B^{\ast }}=\{v\in V^{\prime }~|~Jv~$constant in the
direction$~\beta _{1}\}$. Since $Jv=0$ on the boundary $\partial
\Omega $ then $K_{B^{\ast }}=\{0\}$. Then by using classical
arguments, e.g. \cite{GirRav}, one deduces that $B$ is an
isomorphism from $U $ into $V^{\prime }.$ We pose also
$Y=Z=L^{2}(\Omega ),$ $S=C=I\in \mathcal{L}(Z,Y).$ Now, we introduce
$X=V^{\prime },$ and $\Lambda =J$ which is a nonnegative operator.
The fact that $\Lambda $ is self-adjoint, i.e. that $(\Lambda
v,v^{\prime })_{V^{\prime }}=(v,\Lambda v^{\prime })_{V^{\prime }}$,
comes from the equality $\left\langle \Lambda v,Jv^{\prime
}\right\rangle _{V^{\prime },V}=\left\langle Jv,\Lambda v^{\prime
}\right\rangle _{V,V^{\prime }}.$ To complete the construction, we
pose $\Phi _{U}=(\beta _{1}.\nabla )^{-1},\Phi _{V}=\Lambda ,\Phi
_{V^{\prime }}=I,\Phi _{Y}=\Phi _{Z}=\Lambda ^{1/2}$ which is an
isomorphism from $V^{\prime }$ into $L^{2}(\Omega )$, $a(\Lambda
)=I,\text{ }b(\Lambda )=I\text{ and }c(\Lambda )=I$. Finally, we
proceed as in the first example for the computation of $p_{N,M}$.

\subsection{Example 3: Beam or plate model}

\label{exemple 3}

Here, we deal with a second order problem in time with distributed
internal bounded observation and control.

\subsubsection{The state equation and a choice of semi-decentralized
controllers}

The model under consideration is a fourth order equation posed in a domain $%
\Omega \subset \mathbb{R}^{d}$ which may correspond to a
Euler-Bernoulli clamped beam equation when $d=1$ or to a
Love-Kirchhoff clamped plate equation when $d=2$. The control is
still distributed over the whole domain
and the state is $z:=%
\begin{bmatrix}
w & \partial _{t}w%
\end{bmatrix}%
^{T}$ where $w$ is solution to the boundary value problem
\begin{eqnarray}
\partial _{tt}^{2}w &=&-\Delta ^{2}w+\beta u\text{ in }\mathbb{R}^{+\ast
}\times \Omega ,  \label{eq:etat3} \\
w &=&\nabla w.\nu=0\text{ on }\mathbb{R}^{+\ast }\times \partial
\Omega ,
\label{cL1} \\
w &=&w_{0}\text{ and }\partial _{t}w=w_{1}\text{ in }\Omega \text{
at }t=0, \label{CI1}
\end{eqnarray}%
for a given function $\beta $ and given initial conditions $w_{0}$ and $%
w_{1} $ all defined in $\Omega .$ Choosing the cost functional $\mathcal{J}%
\left( w_{0},w_{1};u\right) =\int_{0}^{+\infty }\left\Vert \Delta
w\right\Vert _{L^{2}\left( \Omega \right) }^{2}+\left\Vert \gamma
u\right\Vert _{L^{2}\left( \Omega \right) }^{2}dt$, we pose $\Lambda
=\left( \Delta ^{2}\right) ^{-1},$ defined as the inverse of the
biharmonic operator
$\Delta ^{2}:H_{0}^{2}(\Omega )\cap H^{4}(\Omega )\rightarrow L^{2}(\Omega )$%
. The method can handle the general case, however in the special case $%
\gamma =\beta =1$, we show in the following Section that the optimal
control
$\uu$ may be approached by 
\begin{equation}
\uu_{N,M,h}=-\frac{1}{2\pi }\sum_{\ell =1}^{M}\omega _{\ell }(\vv%
_{1,h}^{\ell }+\overline{\vv}_{1,h}^{\ell }),  \label{app2ex3}
\end{equation}%
where $(\vv_{1,h}^{\ell },\overline{\vv}_{1,h}^{\ell })$ are
solution to
\begin{equation}
\begin{aligned} \begin{bmatrix} \xi _{1}-\Lambda_{h} & -\xi _{2} \\ \xi _{2}
& \xi _{1}-\Lambda_{h} \end{bmatrix} \begin{bmatrix} \vv_{1,h}^{\ell} \\
\vv_{2,h}^{\ell} \end{bmatrix} &= \begin{bmatrix} \Re e\left( -i\xi
^{\prime }k_{1,N}\left( \xi \right) \right) w_{h} \\ \Im m\left(
-i\xi ^{\prime }k_{1,N}\left( \xi \right) \right) w_{h}
\end{bmatrix}, \\ \begin{bmatrix} \xi _{1}-\Lambda_{h} & -\xi _{2}
\\ \xi _{2} & \xi _{1}-\Lambda_{h}
\end{bmatrix} \begin{bmatrix} \overline{\vv}_{1,h}^{\ell} \\
\overline{\vv}_{2,h}^{\ell} \end{bmatrix} &= \begin{bmatrix} \Re
e\left( -i\xi ^{\prime }k_{2,N}\left( \xi \right) \right)
\partial_tw_{h} \\ \Im m\left( -i\xi ^{\prime }k_{2,N}\left( \xi
\right) \right) \partial_tw_{h}
\end{bmatrix}, \end{aligned}
\end{equation}%
for each quadrature point $\xi _{\ell }:=\xi _{1,\ell }+i\xi
_{2,\ell }$, and
\begin{equation}
\Lambda _{h}^{-1}=\dfrac{1}{h^{4}}%
\begin{bmatrix}
2h^{3} & -\frac{1}{2}h^{3} &        &        &        &                   &      \\
-4     & 6                 & -4     & 1      &        &                   &      \\
1      & \ddots            & 6      & \ddots & \ddots &                   &      \\
       & \ddots            & \ddots & \ddots & \ddots & \ddots            &      \\
       &                   & \ddots & \ddots & 6      & \ddots            & 1    \\
       &                   &        & 1      & -4     & 6                 & -4   \\
       &                   &        &        &        & -\frac{1}{2}h^{3} & 2h^{3}%
\end{bmatrix}%
,  \label{Delta^2h}
\end{equation}%
the vectors $\vv_{i,h}^{T}$, $\overline{\vv}_{i,h}^{T}$,
$w_{h}^{T}$, $\partial _{t}w_{h}^{T} $ being the approximations of $
\begin{pmatrix}
\vv_{i}(x_{1}) & \ldots & \vv_{i}(x_{\mathcal{N}-1})%
\end{pmatrix}
$, $
\begin{pmatrix}
\overline{\vv}_{i}(x_{1}) & \ldots & \overline{\vv}_{i}(x_{\mathcal{N}-1})%
\end{pmatrix}%
$, $
\begin{pmatrix}
w(x_{1}) & \ldots & w(x_{\mathcal{N}-1})%
\end{pmatrix}%
$, $
\begin{pmatrix}
\partial _{t}w(x_{1}) & \ldots & \partial _{t}w(x_{\mathcal{N}-1})%
\end{pmatrix}%
$ and $k_{i,N}$ being defined in the following section.

\subsubsection{Construction and study of the semi-decentralized controllers}

Firstly, the plate equation must be formulated under the form of a
first
order system. We set $z^{T}=%
\begin{bmatrix}
w & \partial _{t}w%
\end{bmatrix}%
$, so we find that $A=%
\begin{bmatrix}
0 & I \\
-\Delta ^{2} & 0%
\end{bmatrix}%
$, the operators $B^{T}=%
\begin{bmatrix}
0 & I%
\end{bmatrix}%
$, $C=%
\begin{bmatrix}
\Delta & 0%
\end{bmatrix}%
,\text{ }S=I$ and the functional spaces $U=L^{2}\left( \Omega
\right) $, $Y\subset L^{2}\left( \Omega \right) $. The usual state
space is $ Z=H_{0}^{2}\left( \Omega \right) \times L^{2}\left(
\Omega \right) $ thus $B$ and $C$ are bounded. We pose
$X=L^{2}\left( \Omega \right) $, $\Lambda =\left( \Delta ^{2}\right)
^{-1}$ an isomorphism from $X$ into $H^{4}\left(
\Omega \right) \cap H_{0}^{2}\left( \Omega \right) $, $\Phi _{Z}=%
\begin{bmatrix}
\Lambda ^{\frac{1}{2}} & 0 \\
0 & I%
\end{bmatrix}%
$, $\Phi _{U}=I$ and $\Phi _{Y}=\Delta \Lambda ^{1/2}$, so $Y=\Delta
\Lambda ^{\frac{1}{2}}L^{2}\left( \Omega \right) =\Delta
H_{0}^{2}\left( \Omega \right) $ and $a\left( \lambda \right) =
\begin{bmatrix}
0 & \lambda ^{-1/2} \\
-\lambda ^{-1/2} & 0%
\end{bmatrix}
$, $b^{T}\left( \lambda \right) =
\begin{bmatrix}
0 & 1%
\end{bmatrix}%
$, $c\left( \lambda \right) =%
\begin{bmatrix}
1 & 0%
\end{bmatrix}%
$ and $s\left( \lambda \right) =1.$

\begin{remark}
\label{isoexp3} $\text{ }$

\begin{enumerate}
\item We indicate how isomorphisms $\Phi _{Y}$ and $\Phi _{Z}$ have been
chosen. The choice of $\Phi _{Z}$ directly comes from the expression
of the inner product $\left( z,z^{\prime }\right) _{Z}=\left( \Phi
_{Z}^{-1}z,\Phi _{Z}^{-1}z^{\prime }\right) _{X^{2}}$ and from
$\left( z_{1},z_{1}^{\prime }\right) _{H_{0}^{2}\left( \Omega
\right) }=\left( \left( \Delta ^{2}\right)
^{\frac{1}{2}}z_{1},\left( \Delta ^{2}\right)
^{\frac{1}{2}}z_{1}^{\prime
}\right) _{L^{2}\left( \Omega \right) }$. For $\Phi _{Y}$, we start from $%
C=\Phi _{Y}c(\Lambda )\Phi _{Z}^{-1}$ and from the relation $\left(
y,y^{\prime }\right) _{Y}=\left( \Phi _{Y}^{-1}y,\Phi
_{Y}^{-1}y^{\prime
}\right) _{X}$ which imply that $\Delta =\Phi _{Y}c_{1}\Lambda^{-\frac{1}{2}%
} $. The expression of $\Phi _{Y}$ follows.

\item The isomorphisms $\Phi _{Z}$ and $\Phi _{U}$ are some matrices of
functions of $\Lambda $, and so $Q$ is also. Thus, the approximation
is
directly developed on $Q=k(\Lambda )$. 
\end{enumerate}
\end{remark}

The controller $Q$ is a $1\times 2$ matrix of operators $k=%
\begin{bmatrix}
k_{1} & k_{2}%
\end{bmatrix}%
$, with $k_{1}=p_{21}\Lambda ^{-\frac{1}{2}}$ and $k_{2}=p_{22}$. So
$\left( k_{i}\right) _{i=1,2}$ is solution to the system
\begin{equation}
\lambda k_{1}^{2}+2k_{1}-1=0 \quad \text{and} \quad
2k_{1}-k_{2}^{2}=0. \label{ex3.numricc}
\end{equation}%
As in the first example, a nonnegative exact solution
\begin{equation*}
k_{1}(\lambda )=\frac{-1+\sqrt{1+\lambda }}{\lambda } \quad
\text{and} \quad k_{2}(\lambda )=\sqrt{2\frac{-1+\sqrt{1+\lambda
}}{\lambda }}
\end{equation*}%
can be exhibited, so it is used to discuss numerical validation.
Again, the functions $k_{i}(\lambda )$ are sufficiently regular to
be accurately approximated by polynomials which computation is done
using the spectral method with Legendre polynomials and the
\textbf{LGL} quadrature formulae.
The weak formulation equivalent to (\ref{ex3.numricc}) is%
\begin{equation}
\int_{I_{\sigma }}\left( \lambda k_{1}^{2}+2k_{1}-1\right) \eta
_{1}(\lambda )\text{ }d\lambda =0 \quad \text{and} \quad
\int_{I_{\sigma }}\left( 2k_{1}-k_{2}^{2}\right) \eta _{2}(\lambda
)\text{ }d\lambda =0 \label{WeakFormEx3}
\end{equation}%
for all $\eta _{1},$ $\eta _{2}\in \mathcal{C}^{0}\left( I_{\sigma
}\right) , $ it is solved by the semi-implicit Algorithm 2 below.

\begin{algorithm}[!h]
\caption{Semi-Implicit scheme for (\ref{WeakFormEx3})}
\begin{algorithmic} [1]

\STATE $ k_{1,N_1}^{0}$, $k_{2,N_2}^{0}$ are given.

\STATE $(m+1)^{\text{th}}$ step: Knowing
$(k_{1,N_{1}}^{m},k_{2,N_{2}}^{m})\in
\mathbb{P}_{N_{1}}\times \mathbb{P}_{N_{2}}$, find $%
(k_{1,N_{1}}^{m+1},k_{2,N_{2}}^{m+1})\in \mathbb{P}_{N_{1}}\times \mathbb{P}%
_{N_{2}}$ such that $\forall \left( \eta _{1},\eta _{2}\right) \in \mathbb{P}%
_{N_{1}}\times \mathbb{P}_{N_{2}}$,
\begin{gather*}
\int_{I_{\sigma}}k_{1,N_1}^{m+1}\left( \lambda \right) \left(
\lambda k_{1,N_1}^{m}\left( \lambda \right) +2\right) \mathit{\eta
}_{1}\left( \lambda \right) d\lambda
=\int_{I_{\sigma}}\eta_{1}\left( \lambda
\right) d\lambda , \\
\int_{I_{\sigma}}k_{2,N_2}^{m+1}\left( \lambda \right) \left(
k_{2,N_2}^{m}\left( \lambda \right) +1\right) \eta_{2}\left( \lambda
\right) d\lambda =\int_{I_{\sigma}}\left( 2k_{1,N_1}^{m+1}\left(
\lambda \right) +k_{2,N_2}^{m}\left( \lambda \right) \right)
\eta_{2}\left( \lambda \right) d\lambda.
\end{gather*}

\STATE If $\|k_{i,N_i}^{m+1}-k_{i,N_i}^{m}\|_{L^2\left(
I_{\sigma}\right)}\leq\varepsilon_i$ then terminate the algorithm
else return to Step 2

\end{algorithmic}
\end{algorithm}

\subsubsection{Numerical results}

The simulations are conducted for a Euler-Bernoulli beam model with length $%
L=4.73m$ so that all eigenvalues of $\Lambda $ are included in
$I_{\sigma }=\left( 0,1\right) $.

\noindent $\vartriangleright $ \textit{Polynomial approximation:}
Numerical
tests show an exponential convergence of the Algorithm 2. For $%
N_{1}=N_{2}=10 $ and for null initial conditions$,$ the exponential
decay rate is about $-1.80$ and this of the differences of
successive iterates $(\left\|
k_{i,N_{i}}^{m+1}-k_{i,N_{i}}^{m}\right\| _{L^{2}\left( I_{\sigma
}\right) })_{i=1,2}$ is about $-1.83$. The two polynomial
approximation errors
\begin{equation*}
e_{i}=\frac{\Vert k_{i,N_{i}}-k_{i}\Vert _{L^{2}(I_{\sigma
})}}{\Vert k_{i}\Vert _{L^{2}(I_{\sigma })}}
\end{equation*}
are in the order of $10^{-10}$ and $10^{-11}$.

\noindent $\vartriangleright $ \textit{Approximation through the
Cauchy formula combined with a polynomial approximation:} The
numerical integrations, carried out as in the first example, yield
relative errors
\begin{equation*}
E_{i}=\frac{\Vert k_{i,N_{i},M}-k_{i}\Vert _{L^{2}(I_{\sigma
})}}{\Vert k_{i}\Vert _{L^{2}(I_{\sigma })}}
\end{equation*}%
parameterized by the number $M$ of integration nodes. They decrease
exponentially with respect to $M$ as shown on Figure
\ref{choixMRex3} where both $M$ are varying from $11$ to $30$.
Several values of the radius $R$ have been tested showing that the
convergence rate is increasing with $R$.

\begin{figure}[!h]
    \centering
    \subfigure[]{\label{microcantilever1} \includegraphics[width=7cm]{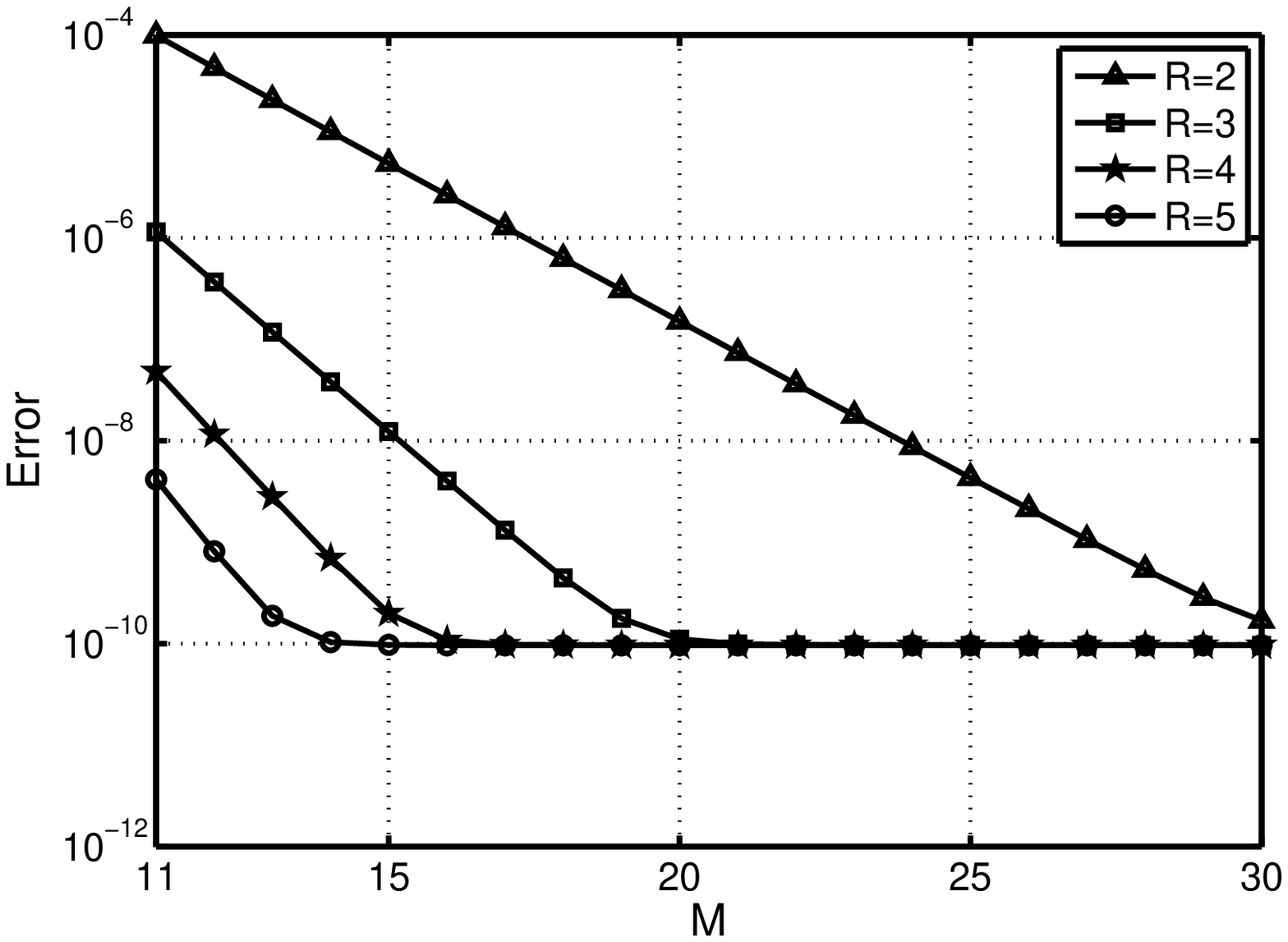}}
    \subfigure[]{\label{microcantilever2} \includegraphics[width=7cm]{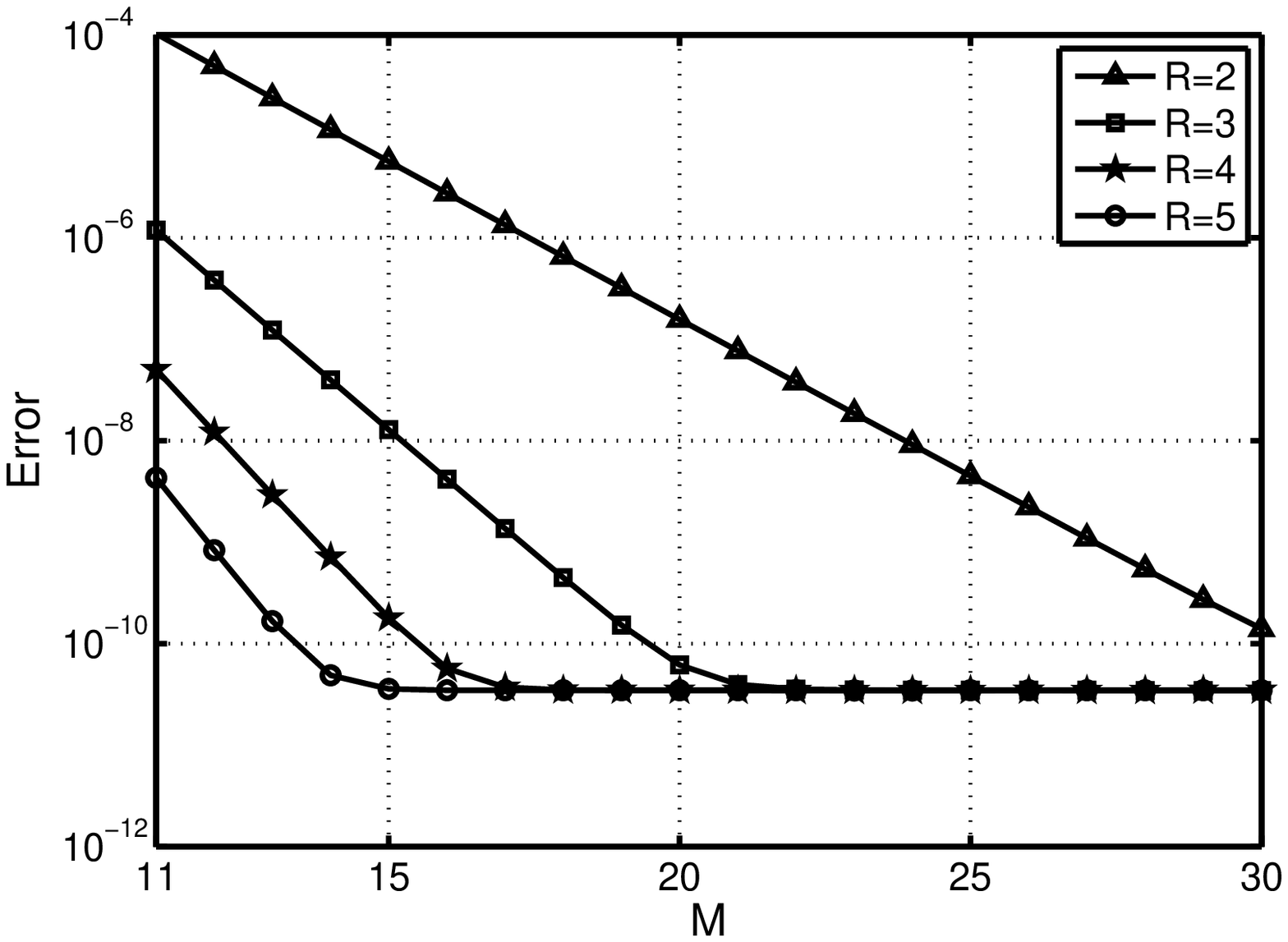}}
\caption{ Relative errors \subref{microcantilever1} $E_{1}$ and
\subref{microcantilever2} $E_{2}$ in logarithmic scale with respect
to $M_{i}$ for different values of $R$ and for $N_{1}=N_{2}=10$}
    \label{choixMRex3}
\end{figure}

\noindent $\vartriangleright $ \textit{Spatial discretization:}
Taking the same notation as in Example 1, the finite difference
discretization of the
one-dimensional fourth order boundary value problem%
\begin{equation}
\Delta ^{2}\vv=f~\text{in }\Omega ,\quad \vv=\nabla \vv.\nu=0~\text{on }%
\partial \Omega ,  \label{BilaplacienDirichlet}
\end{equation}%
is
\begin{equation*}
\frac{1}{h^{4}}\left( \vv_{i-2}-4\vv_{i-1}+6\vv_{i}-4\vv_{i+1}+\vv%
_{i+2}\right) =f({x_{i}}),\quad \text{for }i=2,...,\mathcal{N}-2
\end{equation*}%
for the equation in $\Omega ,$ and $\vv_{0}=0$,
$\vv_{\mathcal{N}}=0$ for the boundary conditions on $\vv$. This
scheme is consistant at the order 2. To do not deteriorate the error
we use a second order scheme for the
boundary conditions on $\partial _{x}$. From Taylor's Theorem, $\vv(x_{1})=%
\vv(0)+h\partial _{x}\vv(0)+\frac{h^{2}}{2}\partial _{xx}^{2}\vv(0)+\mathcal{%
O}\left( h^{3}\right) $ and $\vv(x_{2})=\vv(0)+2h\partial _{x}\vv%
(0)+2h^{2}\partial _{xx}^{2}\vv(0)+\mathcal{O}\left( h^{3}\right) $.
By
eliminating the term in $\partial _{xx}^{2}\vv(0)$ it comes $\partial _{x}\vv%
(0)=\frac{-3\vv(0)+4\vv(x_{1})-\vv(x_{2})}{2h}+\mathcal{O}\left(
h^{2}\right) \text{.}$ The same is done for $\partial _{x}\vv(L)$, we find $%
\partial _{x}\vv(L)=\frac{-3\vv(x_{\mathcal{N}})+4\vv(x_{\mathcal{N}-1})-\vv%
(x_{\mathcal{N}-2})}{2h}+\mathcal{O}\left( h^{2}\right) \text{.}$ In
total, the discretization of the problem
(\ref{BilaplacienDirichlet}) after
elimination of $\vv_{0}$ and $\vv_{\mathcal{N}}$ is written in matrix form $%
[\Delta _{h}^{2}]\vv_{h}=f_{h},$ where $\vv_{h}^{T}=%
\begin{bmatrix}
\vv_{1} & \ldots & \vv_{\mathcal{N}-1}%
\end{bmatrix}%
$, $f_{h}^{T}=%
\begin{bmatrix}
f(x_{1}) & \ldots & f(x_{\mathcal{N}-1})%
\end{bmatrix}%
$ and $[\Delta _{h}^{2}]$ is the matrix in (\ref{Delta^2h}). The
full optimal control approximation (\ref{app2ex3}) is obtained by
using the formulae (\ref{Pbdiscspatiale}-\ref{discspatiale}) and the
formula of $\uu_{N,M,h}$ in (\ref{discspatialeu1}) with $\Lambda
_{h}=[\Delta _{h}^{2}]^{-1}$. To validate this full strategy, we
have carried a computation with $R=5$, $M=11$, $10^{2}$ points in
the mesh of $\Omega $ and for the time $t\in (0,T)$ with $T=15 s$.

The spatial discretization is compared to the expression of the
approximation $\uu_{N,M}(t,x)$ that we calculate thanks to the modal
decomposition of the operator $\partial _{xxxx}^{4}$ with
homogeneous Dirichlet boundary conditions. Its expression is too big
to be presented, it
has been detailed in \cite{Yak}. Denoting by $\uu_{N,M,i}=\left( \uu%
_{N,M,h}\right) _{i}$ the discrete values of the control, the
spatial discretization relative error
\begin{equation*}
\frac{\int_{0}^{T}\left(\sum_{i=0}^{\mathcal{N}}\left| \uu_{N,M,i}(t)-\uu%
_{N,M}(x_{i},t)\right| ^{2}\right)^{\frac{1}{2}}dt}{\int_{0}^{T}\left(\sum_{i=0}^{\mathcal{N}%
}\left\vert \uu_{N,M}(x_{i},t)\right\vert
^{2}\right)^{\frac{1}{2}}dt}
\end{equation*}%
between $u_{N,M}$ and $u_{N,M,h}$ is equal to $1.10 \times 10^{-4}$.

\subsection{Example 4: Two-dimensional heat equation with a boundary control}

\label{exemple 4}This example deals with a special case of boundary
control.

\subsubsection{The state equation and a choice of semi-decentralized
controller}

Let $\Omega $ be the rectangle $\left( 0,1\right) \times \left(
0,\pi \right) \subset \mathbb{R}^{2}$ and $\Gamma _{0}=\left\{
\left( 0,y\right) :0<y<\pi \right\} $ a part of its boundary. Let us
consider the heat equation with a control $v(t,y)$ applied to the
boundary $\Gamma _{0}$,
\begin{equation*}
\left\{ \begin{aligned} \partial _{t}w\left( t,x,y\right)
-\partial_{xx}^2 w\left( t,x,y\right) -\partial _{yy}^2 w\left(
t,x,y\right) &=0 && \text{in }\mathbb{R}^{+\ast }\times \Omega ,
\\ w\left( t,0,y\right) &=v\left(
t,y\right) && \text{on }\mathbb{R}^{+\ast }\times \Gamma _{0}, \\
w\left( t,x,y\right) &=0 && \text{on }\mathbb{R}^{+\ast }\times
\partial \Omega \backslash \Gamma _{0}, \\ w\left(
0,x,y\right)&=w_{0}\left( x,y\right) && \text{in }\Omega \text{.}
\end{aligned}\right.
\end{equation*}%
Since our method is not directly applicable, we reduce the problem
to an internal control problem. We introduce $\overline{w}\left(
t,x,y\right) =w\left( t,x,y\right) -\left( 1-x\right) v\left(
t,y\right) $ solution to the heat equation with homogeneous boundary
conditions,
\begin{equation*}
\left\{ \begin{aligned} \partial _{t}\overline{w}\left( t,x,y\right)
& =\partial _{xx}^2 \overline{w}\left( t,x,y\right) +\partial
_{yy}^2 \overline{w}\left( t,x,y\right) -\left( 1-x\right) u\left(
t,y\right) && \text{in }\mathbb{R}^{+\ast }\times \Omega, \\
\overline{w}\left( t,x,y\right) & = 0 && \text{on }\mathbb{R}^{+\ast
}\times \partial \Omega,
\\ \overline{w}\left( 0,x,y\right) & = \overline{w}_{0}\left( x,y\right)
=w_{0}\left( x,y\right) -\left( 1-x\right) w_{0}\left( 0,y\right) &&
\text{in }\Omega , \end{aligned}\right.
\end{equation*}%
with $u\left( t,y\right) =\partial _{t}v\left( t,y\right) -\partial
_{yy}^{2}v\left( t,y\right) $ that allows for easy computation of $v$ once $%
u $ is known. For simplicity, we define the cost function and the
control space with $u$ instead of $v$. So, we chose the control
space $U\subset
L^{2}\left( \Gamma _{0}\right) $ and the cost functional%
\begin{equation}
\mathcal{J}\left( \overline{w}_{0};u\right) =\int_{0}^{+\infty
}\left\Vert \overline{w}\left( t,x,y\right) \right\Vert
_{L^{2}\left( \Omega \right) }^{2}+\left\Vert u\left( t,y\right)
\right\Vert _{L^{2}\left( \Gamma _{0}\right) }^{2}\text{ }dt\text{.}
\label{fonctionnelle4}
\end{equation}%
Then, the approximation of the control is done by using $J$ terms in
a modal decomposition of $\partial _{xx}^{2}$. Without entering into
much details, that are given in sub-Section \ref{ex4 Appli}, the
state vector is comprised
with $J$ components $\overline{w}_{j}\left( t,y\right) =\sqrt{2}\int_{0}^{1}%
\overline{w}\left( t,x,y\right) \sin \left( j\pi x\right) $ $dx$ and
the associated control is therefore $\uu=-k(\Lambda )z$ where $k$ is
a $J$-row vector of functions and $\Lambda $ is the isomorphism
$(-\partial _{yy}^{2})^{-1}:L^{2}(\Gamma _{0})\rightarrow
H^{2}(\Gamma _{0})\cap H_{0}^{1}(\Gamma _{0})$. A semi-decentralized
control is built from a rational approximation $k_{j,N_{j}}$ of each
component $k_{j}$ and from a
quadrature rule in the Cauchy formula,%
\begin{equation}
\uu_{N,M,h}=-\frac{1}{2\pi }\sum_{\ell =1}^{M}\omega _{\ell }\sum_{j=1}^{J}%
\vv_{1,h}^{\ell ,j},  \label{app4ex4}
\end{equation}%
where each $\vv_{1,h}^{\ell ,j}$ is solution to a system like
(\ref{vh1d})
with $\overline{w}_{j,h}$ instead of $w_{h}$ and $k_{j,N_{j}}$ instead of $%
p_{N}.$

\subsubsection{Construction and study of the semi-decentralized controller}

\label{ex4 Appli}

We start with projecting the model on the $J$ first components of
the orthonormal basis $\psi _{j}\left( x\right) =\sqrt{2}\sin \left(
j\pi x\right) $ in $L^{2}(0,1)$. Since $\int_{0}^{1}\left(
1-x\right) \psi _{j}(x)~dx=\frac{\sqrt{2}}{j\pi }$ the components
$\overline{w}_{j}\left( t,y\right) $ are solution to the equations
posed on $\Gamma _{0}$,
\begin{equation}
\left\{ \begin{aligned} \partial _{t}\overline{w}_{j}\left(
t,y\right) & = -j^{2}\pi ^{2}\overline{w}_{j}\left( t,y\right)
+\partial _{yy}^{2}\overline{w}_{j}\left( t,y\right)
-\frac{\sqrt{2}}{j\pi }u\left(
t,y\right) && \text{in }\mathbb{R}^{+\ast }\times \Gamma_0 , \\
\overline{w}_{j}\left( t,0\right) &=\overline{w}_{j}\left(
t,\pi\right)=0 && \text{in }\mathbb{R}^{+\ast } , \\
\overline{w}_{j}\left( 0,y\right) &=\overline{w}_{j,0}\left(
y\right) =\int_{0}^{1}\overline{w}_{0}\left( x,y\right) \psi
_{j}\left( x\right) dx && \text{in }\Gamma_0 \text{.}
\end{aligned}\right.  \label{Etat4}
\end{equation}%
This is the system of state equations coupled by a common internal control $%
u\left( t,y\right) $. The cost functional (\ref{fonctionnelle4}) is
reduced to
\begin{equation*}
\mathcal{J}\left( \overline{w}_{0};u\right) \simeq \mathcal{J}\left(
\overline{w}_{.,0};u\right) =\int_{0}^{+\infty }\sum_{j=1}^{J}||\overline{w}%
_{j}\left( t,y\right) ||_{L^{2}\left( \Gamma _{0}\right)
}^{2}+||u\left( t,y\right) ||_{L^{2}\left( \Gamma _{0}\right)
}^{2}~dt.
\end{equation*}%
Then, the state variable is $z^{T}=%
\begin{bmatrix}
\overline{w}_{1} & \ldots & \overline{w}_{J}%
\end{bmatrix}%
$, $A=-\diag[(j^{2}\pi ^{2}+\Lambda ^{-1})_{j=1\ldots J}]$, $B^{T}=\frac{\sqrt{2}%
}{\pi }%
\begin{bmatrix}
\frac{I}{1} & \ldots & \frac{I}{J}%
\end{bmatrix}%
$, and $C$ is the identity operator. The control and the observation
spaces are $U=L^{2}\left( \Gamma _{0}\right) $ and $Y=\left(
L^{2}\left( \Gamma _{0}\right) \right) ^{J}$. In addition we pose
$X=L^{2}\left( \Gamma _{0}\right) $ and the state space $Z=\left(
L^{2}\left( \Gamma _{0}\right) \right) ^{J}$ thus $B$ and $C$ are
bounded. Thus $\Phi _{Z}=\Phi _{Y}=I_{J\times J}$, $\Phi _{U}=I$ and
$a\left( \lambda \right) =-\diag\left[ (j^{2}\pi
^{2}+\frac{1}{\lambda })_{j=1\ldots J}\right] $, $b^{T}\left(
\lambda
\right) =\frac{\sqrt{2}}{\pi }%
\begin{bmatrix}
\frac{1}{1} & \ldots & \frac{1}{J}%
\end{bmatrix}
$ and $c\left( \text{.}\right) $ is the identity operator on $\mathbb{R}%
^{J}$. Since $\Phi _{Z}$ and $\Phi _{U}$ are the identity operators,
the approximation is developed on $Q=k(\Lambda )$ with $k(\Lambda )=q(\Lambda )$%
, and the exact optimal control is $\uu=-k(\Lambda )z$.

To build a rational interpolation $k_{N}\left( \lambda \right) $ of
the form (\ref{apprxfractional}) the interval $I_{\sigma }=(0,1]$ is
meshed with $L+1$ distinct nodes $\lambda _{0},\ldots ,\lambda _{L}$
and each $p\left( \lambda _{n}\right) $ solutions to the
\textbf{ARE} is accurately computed with a standard solver. The
exact expression of $k\left( \lambda _{n}\right) =b^{T}p\left(
\lambda _{n}\right) $ follows and the coefficients of the
rational approximation are solution to the $L+1$ equations $%
k_{j,N_{j}}\left( \lambda _{n}\right) =k_{j}\left( \lambda
_{n}\right) $ i.e. to
\begin{equation*}
\sum\limits_{m=0}^{N_{j}^{N}}d_{m}\lambda _{n}^{m}-k_{j}\left(
\lambda _{n}\right) \sum_{m^{\prime }=0}^{N_{j}^{D}}d_{m^{\prime
}}^{\prime }\lambda _{n}^{m^{\prime }}=0\text{ for }n=0,\ldots ,L.
\end{equation*}%
The number $L$ of equations is taken sufficiently large so that the
system with $N^{N}+N^{D}+2$ unknowns is over-determined and is
solved in the mean square sense by using the singular value
decomposition.

\subsubsection{Numerical results}

The simulation have been conducted with four modes i.e. for $J=4$.
The shape
of the four first functions $k_{j}(\lambda )$ are represented in Figure \ref%
{interpo-fraction4} which shows that they exhibit a singular
behavior at the origin. Thus, they can not be accurately
approximated by polynomials but may be by rational functions.

\begin{figure}[h]
\begin{center}
{\normalsize
\scalebox{0.7}{\includegraphics*{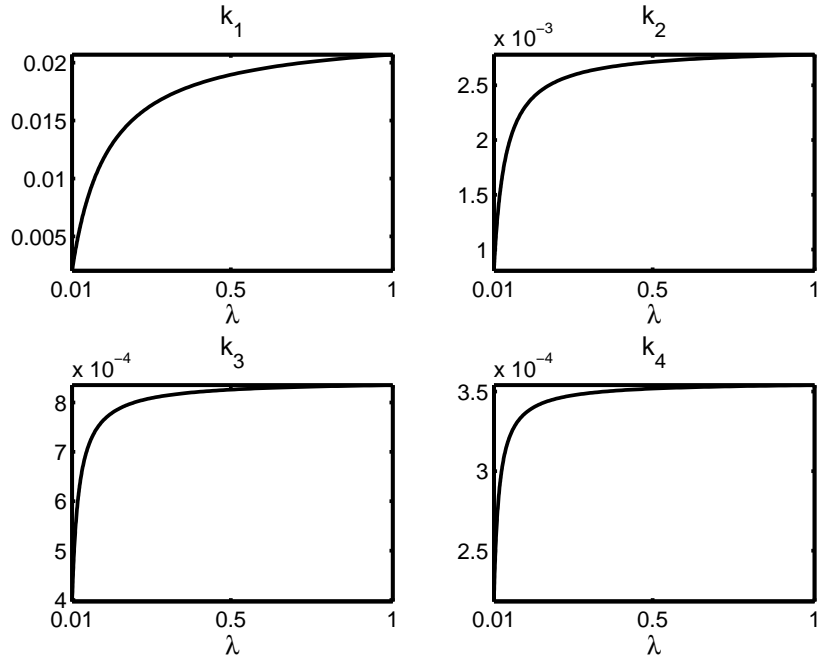}}
}
\end{center}
\caption{Shapes of the Spectral Functions $k$}
\label{interpo-fraction4}
\end{figure}

\medskip

\noindent $\vartriangleright $ \textit{Rational approximation:} In
order to get an accurate approximation, we choose a logarithmic
distribution of $100$ nodes in $(10^{-2},1)$, which corresponds to a
truncation of high frequencies. In Table \ref{ErPNMex4}, we report
the relative errors in the discrete $\ell ^{2}$-norm on the other
set $\{\lambda_{n}\}_{n=0 \ldots 200}$
\begin{equation*}
e_{j}=\frac{\left(\sum_{n=0}^{200} \left| k_{j,N_{j}}\left(
\lambda_{n} \right)-k_{j}\left( \lambda_{n}
\right)\right|^2\right)^{\frac{1}{2}}}{\left(\sum_{n=0}^{200}\left|
k_{j}\left( \lambda_{n} \right)\right|^2\right)^{\frac{1}{2}}},
\text{ with } \lambda_{n} = 10^{-2 + \frac{n}{100}} \quad \text{for
} j=1,\ldots,4 ,
\end{equation*}
between the exact function $k_{j}$ and its rational approximation
$k_{j,N_{j}}$ for special values of numerator's and denominator's
polynomial degrees $N_{j}=\left( N_{j}^{N},N_{j}^{D}\right) $

\begin{table}[h!]
\begin{center}
\caption{Errors of the rational approximations with numerator's and
denominator's degrees $N_{j}=\left( N_{j}^{N},N_{j}^{D}\right) $}
\begin{tabular}{| c | c | c | c | c |}
\hline
$j$                    & $1 $                   & $2 $                   & $3 $                   & $4 $ \\
\hline
$N_{j}$                & $\left( 19,3\right) $ & $\left( 18,3\right) $ & $ \left(17,1\right) $ & $\left( 20,2\right) $ \\
\hline
$e_{j}\times 10^{-10}$ & $0.003$                 & $0.013$                 & $1.003$                 & $0.182$ \\
\hline
\end{tabular}%
\label{ErPNMex4}
\end{center}
\end{table}

\noindent $\vartriangleright $ \textit{The Cauchy formula combined
with rational approximations:} Then, according to Remark \ref{Rq
ellipse}, numerical integrations are performed with a standard
trapezoidal quadrature rule along the ellipse defined by the two
radii in the real and imaginary directions $R_{1}=1.02$ and $R_{2} =
0.07$. The relative errors%
\begin{equation*}
E_{j}=\frac{\left(\sum_{n=0}^{200} \left| k_{j,N_{j},M}\left(
\lambda_{n} \right)-k_{j}\left( \lambda_{n}
\right)\right|^2\right)^{\frac{1}{2}}}{\left(\sum_{n=0}^{200}\left|
k_{j}\left( \lambda_{n} \right)\right|^2\right)^{\frac{1}{2}}}
\text{ with } \lambda_{n} = 10^{-2 + \frac{n}{100}} \quad \text{for
} j=1,\ldots,4,
\end{equation*}
between the exact functions and final approximations are plotted in
logarithmic scale in Figure \ref{ellipse4} for $M$ varying from $10$
to $5 \times 10^{2}$. The errors converge exponentially with an
exponential decay rate given in Figure \ref{ellipse4}. Note that the
parameters $R_1$ and $R_2$ of the ellipse affects the rate of
convergence errors, which is confirmed by our numerical calculation.

\begin{figure}[h]
\begin{center}
{\normalsize
\scalebox{0.6}{\includegraphics*{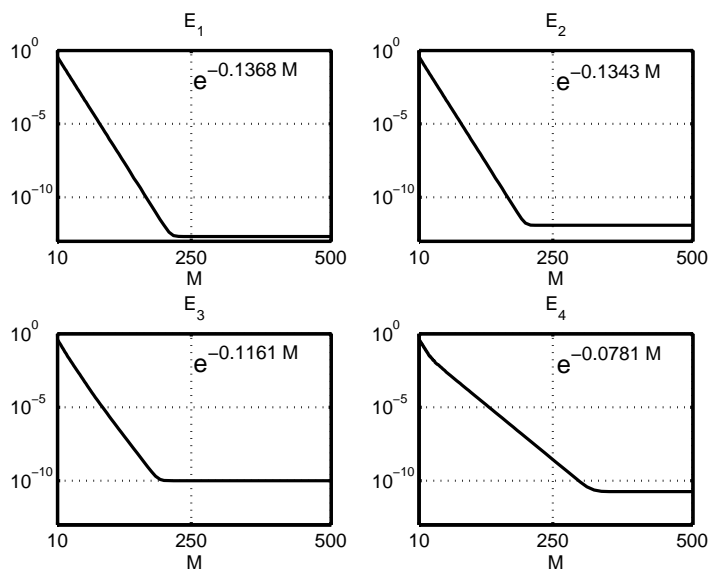}}}
\end{center}
\caption{Errors between $k$ and $k_{N,M}$} \label{ellipse4}
\end{figure}

\medskip

\noindent $\vartriangleright $ \textit{Spatial discretization:} The
approximation (\ref{app4ex4}) is obtained from the formula of
$\uu_{N,M,h}$ in (\ref{discspatialeu1}) and by using the centered
finite difference scheme
of the second order derivative $\partial _{yy}^{2}$. The expression of $%
\Lambda _{h}$ is the same as in Example 1 and the error between
$\uu_{N,M,h}$ and $\uu$ is quadratic in the space step $h$.

\section{Conclusion\label{conclusion}}

\noindent We have proposed a method to compute distributed control
applied to linear distributed systems with a control operator that
is bounded or not. It has been conceived for architectures of
semi-decentralized processors. Its construction uses a functional
calculus for matrices of functions of an operator, based on spectral
theory and Cauchy formula. In the case of polynomial approximation
of $k$, we have noticed that the numerical integration needs few
integration points, and that the radius of the contour affects the
accuracy of the numerical integration of the Cauchy formula. If the
approximation is rational, we have concluded that numerical
integration requires more integration points in the ellipse which
parameters have been chosen heuristically. We think that the
performance of the method could be further improved by finding
optimal contour parameters depending on the number of quadrature
nodes following the ideas in J. A. C. Weideman and L. N. Trefethen
\cite{WeiTre}. Finally, the method can be extended to other
frameworks for distributed control and for functional calculus.


\bibliographystyle{plain}
\bibliography{biblio1}

\begin{thebibliography}{10}

\bibitem{BamPag}
B.~Bamieh, F.~Paganini, and M.A. Dahleh.
\newblock Distributed control of spatially invariant systems.
\newblock {\em IEEE Transactions on Automatic Control}, 47(7):1091--107, 2002.

\bibitem{BanIto}
H.~T. Banks and K.~Ito.
\newblock Approximation in lqr problems for infinite-dimensional systems with
  unbounded input operators.
\newblock {\em J. Math. Systems Estim. Control}, 7(1):p. 34 pp., 1997.

\bibitem{BerMad1991}
C.~Bernardi and Y.~Maday.
\newblock Some spectral approximations of one-dimensional fourth-order
  problems, in: P. nevai and a. pinkus, eds.
\newblock {\em In Progress in approximation theory}, page 43116, 1991.

\bibitem{BerMad1992}
C.~Bernardi and Y.~Maday.
\newblock {\em Approximations spectrales de probl\`{e}mes aux limites
  elliptiques}.
\newblock Mathematiques et Applications 10. Springer-Verlag, (1992).

\bibitem{CroMig}
M.~Crouzeix and A.~L. Mignot.
\newblock {\em Analyse num\'{e}rique des \'{e}quations diff\'{e}rentielles}.
\newblock Collection Math\'{e}matiques Appliqu\'{e}es pour la Maitrise. Masson,
  Paris, 1984.

\bibitem{CurZwa}
R.~F. Curtain and H.~Zwart.
\newblock {\em An introduction to infinite-dimensional linear systems theory},
  volume~21 of {\em Texts in Applied Mathematics}.
\newblock Springer-Verlag, 1995.

\bibitem{DanDul}
R.~D'Andrea and G.~E. Dullerud.
\newblock Distributed control design for spatially interconnected systems.
\newblock {\em IEEE Trans. Automat. Control}, 48(9):1478--1495, 2003.

\bibitem{DauLio}
R.~Dautray and J.-L. Lions.
\newblock {\em Mathematical analysis and numerical methods for science and
  technology}, volume~3.
\newblock Springer-Verlag, Berlin, 1990.

\bibitem{DavRab}
P.~J. Davis and P.~Rabinowitz.
\newblock {\em Methods of Numerical Integration}.
\newblock Computer Science and Applied Mathematics. Academic Press Inc.,
  Orlando, FL,, edition, 1984.

\bibitem{GirRav}
V.~Girault and P.-A. Raviart.
\newblock {\em Finite element methods for Navier-Stokes equations}, volume~5 of
  {\em Springer Series in Computational Mathematics}.
\newblock Springer-Verlag, Berlin, 1986.
\newblock theory and algorithms.

\bibitem{Gri}
P.~Grisvard.
\newblock {\em Elliptic problems in nonsmooth domains}, volume~24 of {\em
  Monographs and Studies in Mathematics}.
\newblock Pitman Advanced Publications Program, Boston-London-Melbourne, 1985.
\newblock theory and algorithms.

\bibitem{Haa}
M.~Haase.
\newblock {\em The functional calculus for sectorial operators}, volume 169 of
  {\em Operator Theory: Advances and Applications}.
\newblock Birkh�user Verlag, Boston, 2006.

\bibitem{HuiYak}
H.~Hui, Y.~Yakoubi, M.~Lenczner, and Ratier.
\newblock Semi-decentralized approximation of a lqr-based controller for a
  one-dimensional cantilever array.
\newblock {\em 18th IFAC World Congress, August 28 - September 2, 2011, Milano,
  Italy}.

\bibitem{Jov}
M.~R. Jovanovi\'{c}.
\newblock On the optimality of localized distributed controllers.
\newblock {\em Int. J. Systems, Control and Communications}, 2(1/2/3):82--99,
  2010.

\bibitem{KadLenMr}
M.~Kader, M.~Lenczner, and Z.~Mrcarica.
\newblock Approximation of an optimal control law using a distributed
  electronic circuit: application to vibration control.
\newblock {\em Comptes Rendus de l'Academie des Sciences Serie II b/Mecanique},
  328(7):547 -- 53, 2000.

\bibitem{LanDAn}
C.~Langbort and R.~D'Andrea.
\newblock Distributed control of spatially reversible interconnected systems
  with boundary conditions.
\newblock {\em SIAM J. Control Optim.}, 44(1):1--28, 2005.

\bibitem{Lasiecka}
I.~Lasiecka and R.~Triggiani.
\newblock {\em Control theory for partial differential equations: continuous
  and approximation theories. I}, volume~74 of {\em Encyclopedia of Mathematics
  and its Applications}.
\newblock Birkh\"{a}user Verlag, 2000.

\bibitem{LenMonYak}
M.~Lenczner, G.~Montseny, and Y.~Yakoubi.
\newblock Diffusive realizations for solutions of some operator equations.
\newblock {\em Math. Comput.}, 81(277):319--344, 2012.

\bibitem{LenYak1}
M.~Lenczner and Y.~Yakoubi.
\newblock Semi-decentralized approximation of optimal control for partial
  differential equations in bounded domains.
\newblock {\em Comptes Rendus M\'{e}canique}, 337(4):245--250, 2009.

\bibitem{Lio}
J.-L. Lions.
\newblock {\em Optimal control of systems governed by partial differential
  equations}.
\newblock Die Grundlehren der mathematischen Wissenschaften, Band 170.
  Springer-Verlag, 1971.

\bibitem{MarSan}
C.~Martinez~Carracedo and M.~Sanz~Alix.
\newblock {\em The theory of fractional powers of operators}, volume 187 of
  {\em North-Holland Mathematics Studies}.
\newblock North-Holland Publishing Co., 2001.

\bibitem{PagBam}
F.~Paganini and B.~Bamieh.
\newblock Decentralization properties of optimal distributed controllers.
\newblock {\em Proceedings of the IEEE Conference on Decision and Control},
  2:1877--1882, 1998.

\bibitem{SanSan}
J.~Sanchez~Hubert and E.~Sanchez-Palencia.
\newblock {\em Vibration and coupling of continuous systems: asymptotic
  methods}.
\newblock Springer-Verlag, 1989.

\bibitem{WeiTre}
J.~A.~C. Weideman and L.~N. Trefethen.
\newblock Parabolic and hyperbolic contours for computing the bromwich
  integral.
\newblock {\em Math. Comp.}, 76(259):1341--1356, 2007.

\bibitem{Yak}
Y.~Yakoubi.
\newblock Two approximation methods for semi-decentralized optimal control of
  distributed systems.
\newblock {\em PhD Thesis, University de Franche-Comt\'e, July 15th 2010}.

\bibitem{Yos}
K.~Yosida.
\newblock {\em Functional analysis}.
\newblock Classics in Mathematics. Springer-Verlag, 1995.
\newblock Reprint of the sixth (1980) edition.

\end{thebibliography}

\end{document}